\documentclass[12pt]{article}
 
 \usepackage{amsmath, latexsym, amsfonts, amssymb, amsthm, amscd, enumerate, comment, stackrel, mathrsfs}
 \usepackage[usenames,dvipsnames,svgnames,table]{xcolor}
 \usepackage[left=3cm, right=2.5cm, top=2.5cm, bottom=2.5cm]{geometry}

 \newtheorem{theorem}{Theorem}[section]
 
 \newtheorem{lemma}[theorem]{Lemma}
 
 \theoremstyle{definition}

 \newtheorem{assumption}[theorem]{Assumption}

 \title{Speed of extinction for CSBP in a subcritical L\'evy environment:  strongly and intermediate cases}
 \author{Natalia Cardona-Tob\'on and Juan Carlos Pardo}

\begin{document}
 	
 	\begin{center}
 		{\Large \bf Speed of extinction for continuous state branching processes in subcritical L\'evy environments: \\[3mm]  the strongly and intermediate regimes}\\[5mm]
 		
 		\vspace{0.7cm}
 		\textsc{Natalia Cardona-Tob\'on\footnote{Institute of Mathematical Stochastics, University of Göttingen. Goldschmidtstrasse 7, 37077 Göttingen, Germany,  {\tt natalia.cardonatobon@uni-goettingen.de}}}  and \textsc{Juan Carlos Pardo\footnote{Centro de Investigación en Matemáticas. Calle Jalisco s/n. C.P. 36240, Guanajuato, México, {\tt jcpardo@cimat.mx}}}

 	\end{center}
 	
 	\vspace{0.3cm}

 	\begin{abstract}
 		\noindent 
 		In this paper, we study  the  speed  of extinction of continuous state branching processes in subcritical L\'evy environments.  More precisely, when the associated L\'evy  process to the environment drifts to $-\infty$ and, under a suitable exponential martingale  change of measure  (Esscher transform), the environment either drifts to $-\infty$ or oscillates.  We extend recent results of Palau et al. \cite{palau2016asymptotic} and Li and Xu \cite{li2018asymptotic}, where the branching term is associated to a  spectrally positive stable L\'evy process  and complement the recent article of Bansaye et al. \cite{bansaye2021extinction} where the critical case was studied. Our methodology combines a path analysis of the branching process together with its L\'evy environment,   fluctuation theory for L\'evy processes and  the asymptotic behaviour of exponential functionals of L\'evy processes. As an application of the aforementioned results, we characterise the process conditioned to survival also known as the $Q$-process.

 	\end{abstract}

 	\section{Introduction and main results}
 	
 	Continuous state branching processes in random environments (or CBPREs for short) are the continuous analogue, in time and space,  of Galton-Watson processes in random environments (or GWREs for short).  Roughly speaking, a process in this class is a time-inhomogeneous Markov process taking values in $[0,\infty]$, where $0$ and 
 	$\infty$ are absorbing states, satisfying the quenched branching property; that is  conditionally on the environment,  the process started from $x+y$ is distributed as the  sum of two independent copies of the same process but issued from $x$ and $y$, respectively.
 	
 	CBPREs  provide a richer class of branching models which take into account the effect of the environment on demographic parameters and letting new phenomena 
 	appear.  In particular, the  classification of the asymptotic behaviour of rare events, such as the survival  and explosion probabilities, is much more complex than the case when the environment is fixed since it may combine environmental and demographical stochasticities. Another interesting feature of CBPREs is that they  also appear as scaling limits of GWREs,  a very rich family of   population models; see for instance  Kurtz \cite{kurtz1978diffusion} where the  continuous path setting is considered and  Bansaye and Simatos \cite{bansaye2015scaling} and Bansaye et al. \cite{bansaye2021extinction}  where  different classes of processes in random environment are studied including CBPREs.   
 	
 	An interesting  family of CBPREs arises  when we consider discrete 
 	population models in i.i.d. environments (see for instance \cite{bansaye2021extinction, bansaye2015scaling, boeinghoff2011branching}). The scaling limit of such population models in i.i.d. environments can be characterised by a stochastic differential equation whose  linear term is driven by a L\'evy process which captures the effect of the environment. 
 	This family of processes is known as 
 	\emph{continuous state branching processes in L\'evy environment} (or CBLEs for short) and 
 	their construction have been given  by He et al. 
 	\cite{he2018continuous} and by Palau and Pardo \cite{palau2018branching}, independently, as the unique non-negative strong solution of a stochastic differential equation which will  be specified below.

 	The study of the long-term behaviour  of  CBLEs  has attracted considerable attention in the last decade due to the interesting properties exhibited by these processes, such as an extra phase transition for the extinction probability in the subcritical regime. A list of key papers includes  Bansaye et al. \cite{bansaye2013extinction}, B\"oinghoff and Hutzenthaler  \cite{boeinghoff2011branching}, He et al.  \cite{he2018continuous}, Palau and Pardo \cite{palau2017continuous, palau2018branching}, Palau et al. \cite{palau2016asymptotic} and Xu \cite{xu2021}. In all these manuscripts, the speed of extinction has been computed for the case 
 	where the associated branching mechanism is either stable  or Gaussian, since the  survival probability can be computed explicitly in terms of  exponential functionals of  L\'evy processes.
 	More precisely,  B\"oinghoff and Hutzenthaler  \cite{boeinghoff2011branching} and Palau and Pardo \cite{palau2017continuous} have studied the case when the random environment is driven by a Brownian motion with drift and when the branching term is given by a Feller diffusion and a stable continuous state branching process, respectively.  Both studies exploited the explicit knowledge of the density of the exponential functional of a Brownian motion with drift. Bansaye et al. \cite{bansaye2013extinction} determined the long-term behaviour for stable branching mechanisms where the 
 		random environment is driven by a L\'evy process with bounded variation paths. The case when the environment is driven by a general L\'evy process satisfying some exponential moments  and the branching mechanism is stable was treated, independently,  by Li and Xu \cite{li2018asymptotic} and 
 		Palau et al. \cite{palau2016asymptotic}. Moreover, the results for the critical regime in the aforementioned two articles can be  extended to the case when the L\'evy environment has not finite second moment  and satisfies the so-called Spitzer's condition (see Theorem 2.20  and Remark 2.21 in Patie and Savov \cite{patie2018bernstein}). More recently, Xu \cite{xu2021} give an exact description for the speed of the extinction probability for CBLEs  with stable branching mechanism and where the L\'evy environment is heavy-tailed.

 Little is known about the long-term behaviour of CBLEs when the associated branching mechanism is neither stable nor Gaussian. Up to our knowledge, the only study in this direction is the recent paper by Bansaye et al.  \cite{bansaye2021extinction}, where  the speed of extinction of  critical CBLEs 
 for more general branching mechanisms  was studied.  More precisely, the authors in  \cite{bansaye2021extinction} considered the case when the associated L\'evy process in  the environmental term  satisfies the so-called Spitzer's condition and relax the assumption that the branching mechanism is  stable. The strategy of their proof relies on the 
 		description of the extinction event under favorable environments, or in other words that   the  running infimum of the environment is positive, and the explicit behaviour of the exponential functional of L\'evy processes under Spitzer's condition given in \cite{patie2018bernstein}.
 		
 		In this article, we are interested in understanding the asymptotic behaviour of the survival probability for CBLEs in the subcritical regime for a more general class of  branching mechanisms rather than the stable case. Recall that in the subcritical regime, the underlying  L\'evy process drifts to $-\infty$. Moreover, as  it was observed in Li and Xu \cite{li2018asymptotic} and  Palau et al. \cite{palau2016asymptotic} and in the discrete case by Afanasyev \cite{afanasyev1980} and Dekking \cite{dekking1987}, there is another phase transition in the subcritical regime. These sub-regimes are known in the literature as:  \textit{strongly}, \textit{intermediate} and \textit{weakly subcritical regimes}, respectively (see e.g.  Theorem 5.1 in \cite{li2018asymptotic} or  Proposition 2.2 in \cite{palau2016asymptotic}). The main contribution of this paper is to provide the precise asymptotic behaviour of the survival probability  in the intermediate and strongly subcritical regimes, under some general assumptions on the L\'evy process associated to the environment and the branching mechanism. Furthermore, we apply our main results to  describe  CBLEs conditioned on survival, also known as $Q$-processes, and we identify them  as CBLEs with immigration  (see for instance Theorem 5.3 in \cite{he2018continuous}, Theorem 1 in \cite{palau2017continuous} or below for a proper definition of the aforementioned class of processes).
 		
 		In the strongly subcritical regime, we deduce that the survival probability  decays exponentially with the same exponential rate as the expected population size (up to a multiplicative constant which is proportional to the initial population size).  The key point in our arguments is to rewrite the probability of  survival under a suitable change of measure which is associated to an exponential martingale of the L\'evy environment. In order to do so, the existence of some exponential moments for the L\'evy environment is required. Under this exponential change of measure, the L\'evy environment remains in the subcritical regime, however the probability of survival now can be related to the Laplace transform of a  CBLRE with immigration. In order to characterise the limit of the survival probability,  we require the so-called Grey's condition which guarantees that the process is absorbed at 0 a.s.(see Corollary 4.4 in \cite{he2018continuous}) and  the characterisation of the Laplace transform of the aforementioned  CBLRE with immigration in terms of an extension of the original environment  to an homogeneous L\'evy process indexed in the real line (see equation 5.6 in \cite{he2018continuous}). The latter characterisation was used in \cite{he2018continuous} to study  the stationary  distribution  of CBLREs with immigration and requires  a classical $x\log x$-moment condition on the L\'evy measure associated to the branching mechanism.  Thus, in order to guarantee the positivity of the limiting coefficient in our result, the $x\log x$-moment condition on the L\'evy measure associated to the branching mechanism is required together with  a $1/x$-moment condition 	on the stationary  distribution of the CBLRE with immigration that appears in the probability of survival.
 		
 		For the intermediate subcritical regime, we obtain that the speed of the survival probability  decays exponentially with a polynomial factor of order ${1/2}$ (up to a multiplicative constant which is proportional to the initial population size). In order to deduce our second main result, we combine  the approach developed in  \cite{afanasyev2014conditional, geiger2003limit}, for the discrete time setting, with  fluctuation theory of L\'evy processes. Similarly as in the strongly subcritical regime, we use an exponential change of measure under which the CBLE now oscillates. In other words, the latter observation allow us to follow a similar strategy developed by Bansaye et al. in \cite{bansaye2021extinction} to study the extinction rate for CBLEs  in the critical regime. More precisely, under this new measure, we split the event of survival in two parts, that is when the running supremum is either negative or positive; and then we show that only paths of the L\'evy process with a very low running supremum give substantial contribution to the speed of survival. In this regime, we impose an $x\log x$-moment condition on the L\'evy measure associated to the branching mechanism  and a   lower bound for the branching mechanism which allow us to control the event of survival under favorable environments. In addition, our arguments require  another technical condition which involves the branching mechanism and the L\'evy process that we will specified below.

 	\subsection{Preliminaries}\label{sec_defandprop}
	Let $(\Omega^{(b)}, \mathcal{F}^{(b)}, (\mathcal{F}^{(b)}_t)_{t\geq 0}, \mathbb{P}^{(b)})$ be a filtered probability space satisfying the usual hypothesis on which we may construct the demographic or branching term of the model that we are interested in. We suppose that $(B_t^{(b)}, t\geq 0)$ is a $(\mathcal{F}_t^{(b)})_{t\geq 0}$-adapted standard Brownian motion and  $N^{(b)}(\mathrm{d} s , \mathrm{d} z, \mathrm{d} u)$ is a  $(\mathcal{F}_t^{(b)})_{t\geq 0}$-adapted Poisson random measure on $\mathbb{R}^3_+$  with intensity $\mathrm{d} s \mu(\mathrm{d} z)\mathrm{d} u$ where  $\mu$ satisfies	\begin{equation}\label{eq_finitemean}
 		\int_{(0,\infty)}(z\wedge z^2)\mu(\mathrm{d}z)<\infty.
 	\end{equation} 
	We denote by $\widetilde{N}^{(b)}(\mathrm{d} s , \mathrm{d} z, \mathrm{d} u)$ for the compensated version of $N^{(b)}(\mathrm{d} s , \mathrm{d} z, \mathrm{d} u)$. Further, we also introduce  the  so-called branching mechanism $\psi$, a convex function with the following L\'evy-Khintchine representation
 \begin{equation}
 		\psi(\lambda) =\psi'(0+) \lambda  + \varrho^2 \lambda^2 + \int_{(0,\infty)} \big(e^{-\lambda x} - 1 + \lambda x \big) \mu(\mathrm{d} x), \qquad \lambda \geq 0,
 	\end{equation}
 where $\varrho \geq 0$. Observe that the term $\psi'(0+)$ is well defined (finite) since condition \eqref{eq_finitemean} holds. Moreover, the function $\psi$ describes the stochastic dynamics of the population.
	
	On the other hand, for the environmental term, we consider another filtered probability space $(\Omega^{(e)}, \mathcal{F}^{(e)},(\mathcal{F}^{(e)}_t)_{t\geq 0}, \mathbb{P}^{(e)})$   satisfying the usual hypotheses. Let us consider $\sigma \geq 0$ and $\alpha$  real constants;  and $\pi$  a measure concentrated on $\mathbb{R}\setminus\{0\}$ such that $$\int_{\mathbb{R}} (1\land z^2)\pi(\mathrm{d} z)<\infty.$$ Suppose that \ $(B_t^{(e)}, t\geq 0)$ \ is a $(\mathcal{F}_t^{(e)})_{t\geq0}$ - adapted standard Brownian motion, $N^{(e)}(\mathrm{d} s, \mathrm{d} z)$ is a $(\mathcal{F}_t^{(e)})_{t\geq 0}$ - Poisson random measure on $\mathbb{R}_+ \times \mathbb{R}$ with intensity $\mathrm{d} s \pi(\mathrm{d} z)$, and $\widetilde{N}^{(e)}(\mathrm{d} s, \mathrm{d} z)$ its compensated version. We denote by $S=(S_t, t\geq 0)$ a L\'evy process, that is  a process with  stationary and independent increments and  c\`adl\`ag paths, with  the following L\'evy-It\^o decomposition
 	\begin{equation*}\label{eq_ambLevy}
 		S_t = \alpha t + \sigma B_t^{(e)} + \int_{0}^{t} \int_{(-1,1)} (e^z - 1) \widetilde{N}^{(e)}(\mathrm{d} s, \mathrm{d} z) + \int_{0}^{t} \int_{(-1,1)^c} (e^z - 1) N^{(e)}(\mathrm{d} s, \mathrm{d} z).
 	\end{equation*}
 Note that   $S$ is a L\'evy process with no jumps smaller  than -1.

 	In our setting, the population size has no impact on the evolution of the environment or in other words we are considering independent processes for the demographic and  environmental terms. More precisely, we work now on the space $(\Omega, \mathcal{F}, (\mathcal{F}_t)_{t\geq 0}, \mathbb{P})$ the direct product of the two probability spaces defined above, that is to say, $\Omega := \Omega^{(e)} \times \Omega^{(b)}, \mathcal{F}:= \mathcal{F}^{(e)}\otimes \mathcal{F}^{(b)},  \mathcal{F}_t:=  \mathcal{F}^{(e)}_t \otimes  \mathcal{F}^{(b)}_t$ for $t\geq0$, $ \mathbb{P}:=\mathbb{P}^{(e)} \otimes \mathbb{P}^{(b)} $. Therefore $(Z_t, t\geq 0)$,
   the \textit{continuous-state branching process  in the L\'evy environment $(S_t, t\geq 0)$} is defined on  $(\Omega, \mathcal{F}, (\mathcal{F}_t)_{t\geq 0}, \mathbb{P})$ as the unique non-negative strong solution of the following stochastic differential equation
	\begin{equation}\label{CBILRE}
	\begin{split}
 		Z_t = &Z_0 - \psi'(0+) \int_{0}^{t} Z_s \mathrm{d} s + \int_{0}^{t} \sqrt{2\varrho^2 Z_s} \mathrm{d} B_s^{(b)}    \\   &\hspace{4cm} +  \int_{0}^{t} \int_{(0,\infty)} \int_{0}^{Z_{s-}}z \widetilde{N}^{(b)} (\mathrm{d} s, \mathrm{d} z, \mathrm{d} u)+ \int_{0}^{t} Z_{s-} \mathrm{d} S_s.
		\end{split}
 	\end{equation}
 	According to  Theorem 3.1 in He et al. \cite {he2018continuous} or Theorem 1 in Palau and Pardo \cite{palau2018branching}, the equation has pathwise uniqueness and strong solution when ~$|\psi'(0+)|<\infty$. 
	Furthermore, when conditioned on the environment, the process $Z$ inherits the branching property of the underlying CSBP previously defined. Let us denote by ~$\mathbb{P}_z$, for its law starting from $z\ge 0$.
 	
 	The analysis of the process $Z$ is deeply related to the behaviour and fluctuations of the L\'evy process $\xi=(\xi_t, t\ge 0)$, defined as follows
 	\begin{equation}\label{eq_envir2}
 		\xi_t = \overline{\alpha} t + \sigma B_t^{(e)} + \int_{0}^{t} \int_{(-1,1)} z \widetilde{N}^{(e)}(\mathrm{d} s, \mathrm{d} z) + \int_{0}^{t} \int_{(-1,1)^c}z N^{(e)}(\mathrm{d} s, \mathrm{d} z),
 	\end{equation}
 	where
 	\begin{equation*}
 		\overline{\alpha} := \alpha -\psi'(0+)-\frac{\sigma^2}{2} - \int_{(-1,1)} (e^z -1 -z) \pi(\mathrm{d} z).
 	\end{equation*}
 	Note that, both processes $S$ and $\xi$ generate the same filtration. Actually, the process $\xi$ is obtained from $S$, changing only the drift and jump sizes. 
	We denote by $\mathbb{P}^{(e)}_x$,  for the law of the process $\xi$ starting from $x\in \mathbb{R}$ and when $x=0$, we use the notation $\mathbb{P}^{(e)}$ for $\mathbb{P}^{(e)}_0$.  
 	
 		Further, under  condition \eqref{eq_finitemean}, the process $\left(Z_t e^{-\xi_t},  t\geq 0\right)$ is a quenched martingale implying that  for any $t\geq 0$ and $z\geq 0$,
 	\begin{equation}\label{martingquenched}
 		\mathbb{E}_z[Z_t \ \vert \ S]=ze^{\xi_t}, \ \qquad \mathbb{P}_z \ \textrm{-a.s},
 	\end{equation}
 	see Bansaye et al. \cite{bansaye2021extinction}.   In other words, the process $\xi$
 	plays an analogous role as the random walk associated to the logarithm  of the offspring means in the discrete time framework
 	and  leads to the usual classification for the long-term behaviour of branching processes. More precisely, we say that the 
 	process $Z$ is subcritical,  critical or supercritical accordingly as  $\xi$ drifts to $-\infty$, oscillates or drifts to $+\infty$. In this manuscript, we focus on the subcritical regime. 
 	
 	In addition, under  condition \eqref{eq_finitemean}, there is  another quenched martingale associated to  $(Z_t e^{-\xi_t}, t\geq 0)$ which allow us to compute its Laplace transform, see for instance Proposition 2 in \cite{palau2018branching} or  Theorem 3.4 in \cite{he2018continuous}. In order to compute the Laplace transform of $Z_t e^{-\xi_t}$, we first introduce  the unique  positive solution $(v_t(s,\lambda, \xi), s\in [0,t])$ of the following backward differential equation 
 	\begin{equation}\label{eq_BDE}
 		\frac{\partial}{\partial s} v_t(s,\lambda, \xi) = e^{\xi_s} \psi_0(v_t(s, \lambda, \xi)e^{-\xi_s}), \qquad v_t(t,\lambda, \xi) = \lambda,
 	\end{equation}
 	where 
 	\begin{equation}\label{eq_phi0}
 		\psi_0(\lambda)= \psi(\lambda)- \lambda \psi'(0+)=\varrho^2 \lambda^2 + \int_{(0,\infty)} \big(e^{-\lambda x} - 1 + \lambda x\big) \mu(\mathrm{d} x) .
 	\end{equation}
 	Then the process $\left(\exp\{-v_t(s,\lambda, \xi) Z_s e^{-\xi_s}\},  0\le s\le t\right)$ is a quenched martingale implying that for any $\lambda\geq 0$ and $t\geq s\geq 0$, 
 	\begin{equation}\label{eq_Laplace}
 		\mathbb{E}_z\Big[\exp\{-\lambda Z_t e^{-\xi_t}\}\  \Big|\, S, \mathcal{F}^{(b)}_s\Big] = \exp\{-Z_se^{-\xi_s}v_t(s, \lambda, \xi)\}.
 	\end{equation}
  Moreover, we also consider the random semigroup $h_{s,t}(\lambda)=e^{-\xi_s}v_t(s,\lambda e^{\xi_t},\xi)$ which is well defined for all $\lambda\geq 0$ and $s\in [0,t]$ and satisfies
 	\begin{equation}\label{eq_transLaplace}
 		\mathbb{E}_z\left[ e^{-\lambda Z_t}\Big|\, S,  \mathcal{F}^{(b)}_s \right] = \exp\left\{- Z_sh_{s,t}(\lambda)\right\},
 	\end{equation}
 see Theorem 3.4 in \cite{he2018continuous}. According to Section 2 in \cite{he2018continuous}, the mapping $s\mapsto h_{s,t}(\lambda)$ is the  unique positive pathwise solution to the integral differential equation 
 	\begin{equation}\label{eq_backward}
 		h_{s,t}(\lambda)= e^{\xi_t -\xi_s}\lambda - \int_{s}^{t}e^{\xi_r-\xi_s}\psi_0\big(h_{r,t}(\lambda)\big)\mathrm{d} r, \quad \quad 0\leq s\leq t.
 	\end{equation}

 	We close this subsection by introducing  CBLEs with immigration which will play a fundamental role  in our arguments.  Let $b\geq 0$ be a positive constant and $\nu$ a L\'evy measure concentrated on $(0,\infty)$ such that
 		$$ \int_{(0,\infty)} (1\wedge z) \nu(\mathrm{d} z) < \infty.$$
 		
 		 We say that  $X=(X_t, t\geq 0)$ is a \textit{continuous state branching process  with immigration in the L\'evy environment} $S$  if it is the unique non-negative strong solution of the following stochastic differential equation
		 \begin{equation}\label{CBLREwI}
 		 \begin{split}
 		 	X_t &= X_0 - ( \psi'(0+)-b) \int_{0}^{t} X_s \mathrm{d} s + \int_{0}^{t} \sqrt{2\varrho^2 X_s} \mathrm{d} B_s^{(b)} + \int_{0}^{t} \int_{(0,\infty)} z N^{(i)}(\mathrm{d} s, \mathrm{d} z)   \\ &  \hspace{4cm} +  \int_{0}^{t} \int_{(0,\infty)} \int_{0}^{Z_{s-}}z \widetilde{N}^{(b)} (\mathrm{d} s, \mathrm{d} z, \mathrm{d} u)+ \int_{0}^{t} X_{s-} \mathrm{d} S_s,
 		 \end{split}
		 \end{equation}
 		 where $N^{(i)}(\mathrm{d} s, \mathrm{d} z) $ is a Poisson random measure with intensity $\mathrm{d} s\nu(\mathrm{d}z)$  (see Theorem 1 in \cite{palau2017continuous} or Theorem 5.1 in \cite{he2018continuous}).
		 The process $X$ is characterised by the branching mechanism $\psi$ and the immigration mechanism $\eta$ which is given by the Laplace exponent of a subordinator, i.e.,
 		$$ \eta(\lambda) = b\lambda + \int_{0}^{\infty} (1-e^{-\lambda x})\nu(\mathrm{d} x).$$ Let us denote its law starting from $z\ge 0$, by $\mathbb{Q}_z$.

		According to Theorem 5.3 in \cite{he2018continuous}, the law of $X$  is characterised as follows: for any $\lambda\geq 0$ and $t\geq 0$, 
 		 	\begin{equation}\label{eq:laplaceimigration}
 			\mathbb{Q}_z\Big[e^{-\lambda X_t}\Big] = \mathbb{E}^{(e)}\left[\exp\left\{-zh_{0,t}(\lambda) - \int_{0}^{t}\eta(h_{s,t}(\lambda))\mathrm{d} s\right\}\right].
 		\end{equation}
	For our purposes, we are interested in the limiting distribution of $\mathbb{Q}_z(X_t\in {\rm d} y)$, as $t$ goes to $\infty$. The  limiting distribution was derived by  He et al. \cite{he2018continuous} under general assumptions and can be characterised as we explain below.  First, we require an extension of the functional $(v_t(s,\lambda,\xi), s\in [0,t])$ to negative  times. In order to do so,  let us consider an independent copy $(\xi'_t, t \geq 0)$ of  the L\'evy process $\xi$ and construct the time homogeneous L\'evy process $\Xi= (\Xi_t, -\infty<t<\infty)$,  indexed by $\mathbb{R}$, as follows
 	\begin{equation}\label{eq_levytiemponeg}
 		\Xi_t:=\left\{\begin{array}{ll}
		-\lim_{s\downarrow -t}\xi'_{s}& \textrm{for }\,\, t<0,\\
		0 &\textrm{for }\,\, t=0,\\
		\xi_{t} & \textrm{for }\,\, t>0.
		\end{array}
		\right .
 	\end{equation} 
 	Note that the latter definition ensures that the L\'evy process $\Xi$  has c\`adl\`ag paths on $(-\infty, \infty)$ and, in particular, if $\xi_t$ drifts to $-\infty$, as $t\to \infty$, a.s., then $\Xi_t$ drifts to $\infty$, as $t\to -\infty$, a.s. We denote by $\mathbf{P}^{(e)}_x$ for the law of the process $\Xi$ such that $\Xi_0=x\in \mathbb{R}$ and, when $x=0$, we use the notation $\mathbf{P}^{(e)}$ for $\mathbf{P}^{(e)}_0$.

 	With  the definition of  $\Xi$ in hand, we introduce the  map $s\in (-\infty,0]\mapsto v_0(s,\lambda,\Xi)$ as the unique positive pathwise solution of
 	\begin{equation}\label{eq_v_0}
 		v_0(s,\lambda, \Xi)= \lambda - \int_{s}^{0}e^{\Xi_r}\psi_0\big(e^{-\Xi_r}v_0(r,\lambda,\Xi)\big)\mathrm{d} r, \quad \quad s\leq 0.
 	\end{equation}
 	Implicitly, it also follows that for $s\leq 0$ the map $s\mapsto h_{s,0}^{\Xi}(\lambda)=e^{-\Xi_s}v_0(s,\lambda e^{\Xi_0},\Xi)$ is the  unique positive pathwise solution to the equation
 	\begin{equation}\label{eq_BDEhneg}
 		h_{s,0}^{\Xi}(\lambda)= e^{-\Xi_s}\lambda - \int_{s}^{0}e^{\Xi_r-\Xi_s}\psi_0\big(h_{r,0}^{\Xi}(\lambda)\big)\mathrm{d} r, \quad \quad s\leq 0.
 	\end{equation}
	Hence by time homogeneity of the process $\Xi$, we have, for any $\lambda\geq 0$ and $t\geq 0$, that
 	\begin{equation}\label{eq:laplaceimigrationneg}
 		\mathbb{Q}_z\Big[e^{-\lambda X_t}\Big] = \mathbf{E}^{(e)}\left[\exp\left\{-z h_{-t,0}^{\Xi}(\lambda) - \int_{-t}^{0}\eta(h_{s,0}^{\Xi}(\lambda))\mathrm{d} s\right\}\right],
 \end{equation}
 see Section 5 in \cite{he2018continuous} for further details. Finally,   since $\xi$ drifts to $-\infty$ and  under the following $\log x$-moment condition for the L\'evy measure $\nu$, 
\begin{equation}\label{eq_logx}
	\int_{1}^{\infty}  \log u \,\nu(\mathrm{d} u) <\infty,
\end{equation} 
according to Theorem 5.6 in \cite{he2018continuous}, we have that there exists a probability measure $\Pi$ on $[0,\infty)$ such that $\mathbb{Q}_z(X_t\in \cdot)$ converges weakly towards $\Pi$, as $t$ goes to $\infty$, for every $z\ge 0$. Moreover, 
\begin{equation}\label{ergodic}
\int_{[0,\infty)} e^{-\lambda y}\Pi(\mathrm{d} y) = \mathbf{E}^{(e)}\left[\exp\left\{- \int_{-\infty}^{0}\eta(h_{s,0}^{\Xi}(\lambda))\mathrm{d} s\right\}\right].
\end{equation}
It is important to note that  $\psi^\prime $ is the Laplace exponent of a subordinator. Indeed
\[
\psi_0^\prime(\lambda)=2\varrho^2\lambda+\int_{(0,\infty)} (1-e^{-\lambda x})x\mu({\rm d} x),
\]
 so when we take $\eta\equiv \psi^\prime_0$ in the previous discussion, condition \eqref{eq_logx} is reduced to the following assumption that will be used below.
 \begin{assumption}\label{eq_xlogx} The L\'evy measure $\mu $ satisfies
 \begin{equation*}
	\int_{1}^{\infty}  u\log u \,\mu(\mathrm{d} u) <\infty.
\end{equation*}
 \end{assumption}

 For our purposes, we require the following exponential moment condition on the L\'evy environment $\xi$, 
 	 \begin{assumption}\label{eq_moments}
 	 	there exists $\vartheta>0$ such that 
 	 		\begin{equation*}
 	 	  \int_{\{|x|>1\}}e^{\theta x}\pi(\mathrm{d} x)<\infty, \qquad  \textrm{for all} \qquad \theta \in [0, \vartheta].
 	 	\end{equation*}
 	 \end{assumption}
 This condition is equivalent to the existence of  the Laplace transform of $\xi$, that is  $\mathbb{E}^{(e)}[e^{\theta \xi_1}]$ is well defined for $\theta\in [0, \vartheta]$ (see for instance Lemma 26.4 in Sato \cite{ken1999levy}). The latter implies that we can  introduce  the Laplace exponent of $\xi$  as follows $\Phi_\xi(\theta):=\log \mathbb{E}^{(e)}[e^{\theta \xi_1}]$, for $\theta\in [0, \vartheta]$. Again from Lemma 26.4  in \cite{ken1999levy}, we also have $\Phi_\xi(\theta)\in C^\infty$ and $\Phi_\xi^{\prime\prime}(\theta)>0$, for $\theta\in (0, \vartheta)$.

Another object which will be relevant for our analysis  is the so-called exponential martingale associated to  $\xi$, i.e.
 	\[
 	M^{(\theta)}_t=\exp\Big\{\xi_{t}-t\Phi_\xi(\theta)\Big\}, \qquad t\ge 0,
 	\]
 	which is well-defined for $\theta\in [0,\vartheta]$ under assumption \eqref{eq_moments}. It is well-known that $(M^{(\theta)}_t, t\ge 0)$ is  a $(\mathcal{F}^{(e)}_t)_{t\ge 0}$-martingale and that it induces a change of measure which is known as the Esscher transform, that is to say
 	\begin{equation}\label{eq_medida_Essher}
 		\mathbb{P}^{(e,\theta)}(\Lambda):= \mathbb{E}^{(e)}\Big[M_t^{(\theta)} \mathbf{1}_{\Lambda}\Big], \qquad \textrm{for}\quad \Lambda\in \mathcal{F}^{(e)}_t.
 	\end{equation}
 	Similarly to the critical case, which was studied by  Bansaye et al. \cite{bansaye2021extinction}, the asymptotic analysis  of the intermediate subcritical regime requires the notion of the renewal functions $U^{(\theta)}$ and $\widehat{U}^{(\theta)}$ under $\mathbb{P}^{(e,\theta)}$, which are associated to the supremum and infimum of $\xi$, respectively. We refer to Section \ref{prelimLP} for a proper definition or Section VI.4 in Bertoin \cite{bertoin1996levy}. The renewal functions $U^{(\theta)}$ and $\widehat{U}^{(\theta)}$ are  finite, subadditive, continuous and increasing. Moreover, they are identically 0 on $(-\infty, 0]$ and strictly positive on $(0,\infty)$   (see for instance  Lemma 6.4 and Section 8.2 in the monograph of Doney
 	\cite{doney2007fluctuation}). The renewal functions are also related to the so-called ladder height processes which have the same range as the supremum and infimum of $\xi$. In order to state our main results, we require the notation of the  increasing ladder height process $H=(H_t, t\geq 0)$ which is releted to the supremum of $\xi$, see Section \ref{prelimLP} for a proper definition.
 	
\subsection{Main results}\label{sec:mainresults}

We  are now ready to state our main results. 
In order to introduce our main results we require some technical assumptions  on the branching mechanism and on the environment which will control the event of survival. Let us start with  the strongly subcritical regime where some  assumptions on the branching mechanism and  on the environment are required. For the environment,  we assume Assumption \ref{eq_moments} up to $\vartheta =1$, which guarantee the existence of exponential positive moments on $[0,1]$, together with $\Phi_\xi'(0)<0$ and $\Phi_\xi'(1)<0$. The latter guarantees the use of the Esscher transform at  $\theta=1$ in   \eqref{eq_medida_Essher}, and that the L\'evy environment is strongly subcritical.

For  the branching mechanism, we require two conditions: Assumption \ref{eq_xlogx} for the  L\'evy measure $\mu$ and the so-called Grey's condition, i.e.
\begin{assumption}\label{eq:grey} the function $\psi_0$ satisfies
	\begin{equation*}
	\int_{1}^{\infty}\frac{1}{\psi_0(\lambda)} \mathrm{d} \lambda < \infty.
\end{equation*} 
\end{assumption}
The latter assumption guarantees that the process $Z$ is  absorbed at $0$, $\mathbb{P}_z$-a.s., for $z>0$,  see Corollary 4.4 in \cite{he2018continuous}.   It is important to note that  Grey's condition  is a necessary and sufficient condition for absorption, with positive probability,  both for  CSBPs (see Grey \cite{grey1974asymptotic}) and for CBLEs (see Theorem 4.1 in \cite{he2018continuous}).  
On the other hand, the $x\log x$-moment condition for the L\'evy measure $\mu$ is a necessary and sufficient condition for the ergodicity of the CBLE with immigration  which is implicit under the Esscher transform of  $Z$. The $x\log x$-moment condition also appears in the discrete setting  to study the long-term behaviour of branching processes in a strongly subcritical  random environment (see e.g. Theorem 1.1 in \cite{afanasyevstrongly}).

It turns out that, in this regime,   the survival probability  decays with the same rate as the expected generation size, i.e.  as \ $\mathbb{E}_z[Z_t]$ for $t$ large enough,  up to a multiplicative constant. A similar behaviour   appears for subcritical Galton-Watson processes  as well as for discrete branching processes in  random environments in the strongly subcritical regime (see for instance Theorem 1.1 in \cite{afanasyevstrongly}). Moreover from \eqref{martingquenched}, we have 
\begin{equation*}
	\mathbb{E}_z[Z_t]=z\mathbb{E}[e^{\xi_t}] = ze^{\Phi_\xi(1)t}.
\end{equation*}
In other words,  the survival probability decays exponentially up to a multiplicative constant which is proportional to the initial state of the population as is stated below.  

\begin{theorem}\label{theo_stronglycase}
	Suppose that we are in the strongly subcritical regime and that Assumptions \ref{eq_xlogx}  and \ref{eq:grey} are fulfilled. We also assume that  
	\begin{equation}\label{eq:hipstrongcase}
		\int_{0}^{\infty}\mathbf{E}^{(e,1)}\left[\exp\left\{-\int_{-1}^{0} \psi'_0\big(h_{s,0}^{\Xi}(\lambda)\big)\mathrm{d} s\right\}\right] \mathrm{d}\lambda<\infty.
	\end{equation}
	Then for every $z>0$, we have 
	\[\begin{split}
		\lim\limits_{t\to\infty}e^{-\Phi_{\xi}(1)t}\mathbb{P}_{z}(Z_t>0)   = z\mathfrak{B}_1,
	\end{split}\]
	where 	\[\mathfrak{B}_1=  \int_{0}^{\infty}\mathbf{E}^{(e,1)}\left[\exp\left\{-\int_{-\infty}^{0}\psi_0'(h_{s,0}^{\Xi}(\lambda))\mathrm{d} s\right\}\right] \mathrm{d} \lambda\in (0,\infty).\]
\end{theorem}

It is important to note that condition \eqref{eq:hipstrongcase} can be interpreted in terms of a CBLE with immigration. In other words, let $X=(X_t, t\ge 0)$ be a CBLE with immigration starting from 0 with branching and immigration mechanisms $(\psi_0, \psi_0^\prime)$; and with auxiliary L\'evy process (to the environment) given by $(\xi, \mathbb{P}^{(e,1)})$. Let us denote by $\mathbb{Q}_0$ for its law. Hence, condition \eqref{eq:hipstrongcase} is equivalent to 
\[
\mathbb{Q}_0\left[\frac{1}{X_1}\right]<\infty
\]
and 
\[
\mathfrak{B}_1=\int_{[0,\infty)} \frac{1}{y}\,\Pi_{\psi_0}(\rm {d} y),
\]
where $\Pi_{\psi_0}(\rm{d}y)$ denotes the limiting probability distribution of $\mathbb{Q}_0(X_t\in \rm{d}y)$. 

In  general, it seems difficult to compute explicitly the constant $\mathfrak{B}_1$ except for the stable case, that is when    $\psi_0(\lambda)= C\lambda^{1+\beta}$  with $\beta\in (0,1]$ and $C>0$ (observe that $C=2\varrho^2$ when $\beta=1$), where  $\mathfrak{B}_1$ can be computed explicitly and coincides with the constant that appears  in Theorem 5.1 in Li and Xu \cite{li2018asymptotic}.  Note that in this  case $\psi_0$ satisfies Assumptions \ref{eq_xlogx} and \ref{eq:grey} and also we observe that the integral in condition \eqref{eq:hipstrongcase} can be rewritten as follows
			\begin{equation*}
				\begin{split}
					\int_{0}^{\infty}\mathbf{E}^{(e,1)}\left[\exp\left\{-\int_{-1}^{0} \psi'_0\big(h_{s,0}^{\Xi}(\lambda)\big)\mathrm{d} s\right\}\right] \mathrm{d}\lambda&  = \mathbb{E}^{(e,1)}\left[\left(\int_{0}^{1}e^{\beta \xi_u}\mathrm{d} u\right)^{-1/\beta}\right].
				\end{split}
			\end{equation*}
		Now, using the exponential moment Assumption \ref{eq_moments} with $\vartheta=1$ and Lemma 2.2 in \cite{he2018continuous}, we deduce that the expectation in the  right-hand side is finite. Therefore, for any CBLE with stable branching mechanism and associated L\'evy environment satisfying Assumption \ref{eq_moments} with $\vartheta =1$, $\Phi_\xi'(0)<0$ and $\Phi_{\xi}'(1)<0$, we get from Theorem \ref{theo_stronglycase} that
			$$	\lim\limits_{t\to\infty}e^{-\Phi_{\xi}(1)t}\mathbb{P}_{z}(Z_t>0)   = z(\beta C)^{-1/\beta} \mathbb{E}^{(e,1)}\left[\left(\int_{0}^{\infty}e^{\beta \xi_u}\mathrm{d} u\right)^{-1/\beta}\right].   $$
			We refer to subsection \ref{sec:stablecase} for further details  about the computation of  this constant.


Our next main result deals with the intermediate subcritical regime.  This regime is governed by the exponential moment condition on the environment \eqref{eq_moments} with $\vartheta>1$,  together with $\Phi_\xi'(0)<0$ and $\Phi_\xi^\prime(1)=0$.  In other words, the L\'evy process $\xi$ drifts to $-\infty$ under $\mathbb{P}^{(e)}$ and oscillates  under the probability measure $\mathbb{P}^{(e,1)}$, induced by the Esscher transform \eqref{eq_medida_Essher}. In order to state our result, we require to introduce the notion of  L\'evy processes conditioned to stay positive.   According to Lemma 1 in Chaumont and Doney \cite{chaumont2005levy},  the process $(\widehat{U}(\xi_t)\mathbf{1}_{\{\underline{\xi}_t>0\}}, t\geq 0)$ is a martingale with respect to $(\mathcal{F}_t^{(e)})_{t\geq 0}$. With the help of  this martingale,  we define  a new measure which corresponds to the law of $\xi$ \textit{conditioned to stay positive}, as follows: for $\Lambda \in \mathcal{F}^{(e)}_t$ and $x>0$, 
\begin{equation} \label{defPuparrowx}
	\mathbb{P}^{\uparrow}_{x} (\Lambda)
	:=\frac{1}{\widehat{U}(x)}\mathbb{E}^{(e)}_{x}\left[\widehat{U}(\xi_t) \mathbf{1}_{\left\{\underline{\xi}_t> 0\right\}} \mathbf{1}_{\Lambda}\right].
\end{equation}
Similarly, we can define the process conditioned to stay positive under $\mathbb{P}^{(e,1)}$ but  with the martingale associated with the renewal measure $\widehat{U}^{(1)}$ instead of $\widehat{U}$. For a complete overview on this theory, the reader is referred to the monographs of Bertoin \cite{bertoin1996levy} and Doney \cite{doney2007fluctuation}, see also  Chaumont \cite{chaumont1996conditionings} and Chaumont and Doney \cite{chaumont2005levy} and  references therein.

Similarly as in the critical regime studied by Bansaye et al. \cite{bansaye2021extinction}, we require the following assumption on the branching mechanism
\begin{assumption}\label{eq_psiupperbound}
	there exists $\beta \in (0,1]$ and $C>0$ such that 
\begin{equation*}
  \psi_0(\lambda)\geq C \lambda^{1+\beta} \qquad \text{for } \quad \lambda \geq 0.
\end{equation*}
\end{assumption}
This condition will help us to control the probability of survival under  environments with large extrema, that is when  the  supremum of the environment is very large. We also note   that it  implies Grey's condition \eqref{eq:grey} and thus  we ensure that $Z$ gets extinct in finite time with positive probability. 

Furthermore, we also assume that the L\'evy measure $\mu$, which is associated to the branching mechanism, satisfies the $x\log x$ moment condition \eqref{eq_xlogx}.

Similarly as in the strongly case, the survival probability  decays exponentially but now with a  polynomial factor of order $1/2$, up to a multiplicative constant which is proportional to the initial state of the population. In other words, we have the following result.

\begin{theorem}[Intermediate subcritical regime]\label{theo_interm}Suppose that Assumptions \ref{eq_xlogx}, \ref{eq_moments}, with $\vartheta >1$, $\Phi_\xi'(0)<0$, $\Phi_{\xi}'(1)=0$,  and \ref{eq_psiupperbound} hold. 
We also suppose that, for  $x<0$ ,
	\begin{equation}\label{eq:assumpinter}
		\int_{0}^{\infty}\mathbf{E}^{(e,1), \uparrow}_{-x}\left[\exp\left\{-\int_{-1}^{0}\psi_0'(h_{s,0}^{\Xi}(\lambda)) \mathrm{d} s \right\} \right]\mathrm{d} \lambda <\infty.
	\end{equation} 
	Then, for every $z>0$, we have 
	\begin{equation*}
		\lim\limits_{t \to \infty} t^{1/2}e^{-\Phi_{\xi}(1)t} \mathbb{P}_{z}(Z_t>0)=z\mathbb{E}^{(e,1)}\big[H_1\big]\sqrt{\frac{2}{\pi \Phi_\xi''(1)}}\mathfrak{B}_2 ,
	\end{equation*}
	where
	\[\begin{split}
		\mathfrak{B}_2 &= \lim_{x \to -\infty} U^{(1)}(-x) \mathbf{E}^{(e,1), \uparrow}_{-x}\left[\int_{0}^{\infty}\exp\left\{-\int_{-\infty}^{0}\psi_0'\big(h_{s,0}^{\Xi}(\lambda)\big)\mathrm{d} s\right\}\mathrm{d} \lambda\right].
	\end{split}\]
	
\end{theorem}

We conclude this section with our  last main result which is  devoted  to  the study of the CBLE process conditioned to survival or  $Q$-process.  As we will see below and in an analogous sense to what has been proved for the classical Galton-Watson process (see for instance Section 12.3 in \cite{kyprianou2014fluctuations}) and CSBP (see Theorem 4.1 in \cite{lambert2007}), the conditioned process has the same law as a CBLE with immigration but with a different random environment  to the original. In  particular, we extend recent results in Lambert \cite{lambert2007}, for the  case of CSBPs;  and in Palau and Pardo \cite{palau2018branching}, for the continuous state branching  processes in a Brownian environment with a  stable branching mechanism. Let 
	$$ T_0= \{t\geq 0: Z_t =0\},$$
be the extinction time of the process $Z$. 

\begin{theorem}[$Q$-process]\label{teo:Qprocess}
		Let $(Z_t, t\geq 0)$ be a CBLE in a strongly subcritical regime, i.e. satisfying the conditions in Theorem \ref{theo_stronglycase}, or in an intermediate subcritical regime, i.e. satisfying the conditions in Theorem \ref{theo_interm}. Then for all $z, t>0$, 
		\begin{itemize}
		\item[(i)] the conditional laws $\mathbb{P}_z(\ \cdot \ | \ T_0>t)$ converges, as $t\to\infty$, towards a limit law $\mathbb{P}^\natural$, in the sense that for any $t\ge 0$ and $\Lambda\in \mathcal{F}_t$, 
		\[
		\lim_{s\to \infty}\mathbb{P}_z(\Lambda\ |\  T_0>s)=\mathbb{P}^\natural_z(\Lambda).
		\]
		\item[(ii)] The probability measure  $\mathbb{P}^\natural_z$ can be expressed as Doob $h$-transform of  $\mathbb{P}_z$ based on the martingale 
		\[
		D_t=Z_t e^{-\Phi_\xi(1)t}, 
		\]
		that is, for $\Lambda\in \mathcal{F}_t$,
		\[
		 \mathbb{P}_z^{\natural}(\Lambda): = \mathbb{E}_z\left[\frac{Z_t}{z} e^{-\Phi_\xi(1)t}\mathbf{1}_{\Lambda}\right].
		 \]
		 \item[(iii)] The process $Z$, under $\mathbb{P}^\natural_z$, is a CBLE with immigration initiated at $z$, with immigration mechanism given by $\psi_0^\prime(\lambda)$ and the  auxiliary L\'evy process (to the environment) is given by $(\xi, \mathbb{P}^{(e,1)})$. Moreover,  for any $\lambda\geq 0$ and $ t> 0$, we have
	\[
		\mathbb{E}^{\natural}_z\Big[e^{-\lambda Z_t}\Big] = \mathbb{E}^{(e,1)} \left[\exp\left\{-zh_{0,t}(\lambda) -\int_{0}^{t}\psi_0'(h_{s,t}(\lambda))\mathrm{d} s\right\}\right].
	\]

		\end{itemize}
		\end{theorem}

The remainder of this paper is devoted to the proofs of the main results.

	\subsection{Comments about our results} 
Conditions \eqref{eq:hipstrongcase} and \eqref{eq:assumpinter}  seem to be optimal. For instance in  the strongly subcritical  regime, it turns out that  that the probability of survival   $e^{-\Phi_{\xi}(1)t}\mathbb{P}_{z}(Z_t>0) $ is exactly $\mathbb{Q}_z[1/X_t]$, where $(X,\mathbb{Q}_z)$ is a CBLE with immigration starting from $z$ with branching and immigration mechanisms $(\psi_0, \psi_0^\prime)$; and with auxiliary L\'evy process (to the environment) given by $(\xi, \mathbb{P}^{(e,1)})$.  In order to obtain  the limit, some control is required in the  Laplace transform of $X_t$
 and it is here where condition \eqref{eq:hipstrongcase} appears by eliminating the term which depends on the starting population $z$, since it  goes to 0 as $t$ increases. Moreover, we also remark that   condition \eqref{eq:hipstrongcase} is not required in the discrete setting  due to some  monotonicity properties associated to the random walk and the sequence of probability generating functions  (see Lemma 2.1 and 2.3 in \cite{geiger2003limit}), properties that are lost in the continuous setting.

 Regarding Assumption \ref{eq_psiupperbound}, it  seems quite difficult to get rid of it since   some explicit knowledge or properties of the functional $h_{0,t}(\infty)$ are needed to control the behaviour of 
\[
e^{-\xi_t}\mathbb{P}_{z}\big(Z_t>0\ \big|\big.  \ \xi\big), \qquad\textrm{for}\quad z>0,
\]
under favourable environments, as $t$ increases. In other words, we can rewrite the probability of survival in a favourable environment as follows 
	\begin{eqnarray*}
	e^{-t\Phi_\xi(1)}\mathbb{P}_{(z,x)}\left(Z_t>0,\  \sup_{0\leq s\le t}\xi_{s}\geq y\right) &=&  \mathbb{E}_x^{(e,1)}\Big[e^{-\xi_t}\left(1-\exp\big\{-zh_{0,t}(\infty)\big\}\right)\mathbf{1}_{\{ \overline{\xi}_{t}\geq y\}}\Big].
\end{eqnarray*}
In other words to control the right hand side of the previous identity  is somehow quite involved, contrary to the discrete setting where the quenched probability of survival can be upper bounded using a first moment estimate since the event of survival is equal to the event of the current population being bigger or equal to one, an strategy which cannot be used in our setting. 

Finally, we believe that it is possible to obtain  Theorem \ref{theo_interm} with Assumption \ref{eq_moments} but with the less restrictive condition that $\vartheta=1$ with $\Phi_\xi'(0)<0$ and  $\Phi_{\xi}'(1)=0$. It is important to note that even in the case when the branching mechanism is stable, a deeper analysis is required to deduce such result. More precisely, when 
$\psi_0(\lambda)= C\lambda^{1+\beta}$  with $\beta\in (0,1)$ and $C>0$, we have
$$e^{-\xi_t}\mathbb{P}_{z}\big(Z_t>0\ \big|\big.  \ \xi\big)= e^{-\xi_t}\left(1-\exp\left\{-z\left(\beta C\int_{0}^{t}e^{\beta \xi_u}\mathrm{d} u\right)^{-1/\beta}\right\}\right).   $$
Even though the exponential functional of a L\'evy process is a well  studied object, it seems difficult to control the previous random variable, under $\mathbb{P}^{(e,1)}$, with the restriction that $\vartheta=1$ due to the nature of the exponential functional  together with  $e^{-\xi_t}$. 

We conjecture that, under the so-called Spitzer's condition 
$$\frac{1}{t}\int_{0}^{t} \mathbb{P}^{(e,1)}(\xi_s\geq 0)\mathrm{d} s \to \rho \in(0,1), \qquad \text{as} \quad t \to \infty,$$
together with Assumptions  \ref{eq_xlogx}, \ref{eq:grey} and \eqref{eq:assumpinter}, the survival probability must behave as follows:  for  $z>0$,
\begin{equation*}
	e^{-\Phi_{\xi}(1)t} \mathbb{P}_{z}(Z_t>0)\sim z\mathfrak{B}_3t^{-\rho}\ell(t), \quad \textrm{as}\quad t\to \infty,
\end{equation*}
where $\mathfrak{B}_3$ is a positive constant and $\ell$ is a slowly varying function at $\infty$.

\section{Proofs}\label{sec:stronglyregime}
\subsection{Preliminaries on L\'evy processes}\label{prelimLP}
 	In this section, we briefly recall some important facts of L\'evy processes and its fluctuation theory  that we will require in what follows. Recall that  $\mathbb{P}^{(e)}_x$  denotes the law of the L\'evy process $\xi$ 
 	starting from $x\in \mathbb{R}$ and when $x=0$, we use the notation $\mathbb{P}^{(e)}$ for $\mathbb{P}^{(e)}_0$. The dual process $\widehat{\xi}=-\xi$ is also a L\'evy process satisfying that for any fixed time $t>0$, the processes 
 	\begin{equation}\label{eq_lemmaduality}
 		(\xi_{(t-s)^-}-\xi_{t}, 0\le s\le t)\qquad  \textrm{and}\qquad (\widehat{\xi}_s, 0\le s\le t),
 	\end{equation}
 	have the same law, with the convention that $\xi_{0^-}=\xi_0$ (see for instance Lemma 3.4 in Kyprianou \cite{kyprianou2014fluctuations}). For every $x\in \mathbb{R}$,  let $\widehat{\mathbb{P}}_x^{(e)}$ be the law of $x+\xi$ under $\widehat{\mathbb{P}}^{(e)}$, that is the law of $\widehat{\xi}$ under $\mathbb{P}_{-x}^{(e)}$. In the sequel, we assume that $\xi$ is not a compound Poisson process to avoid the possibility  that in this case the process visits the same maxima or minima at distinct times which can make our analysis more involved.
 
 	Let us introduce the running infimum  and supremum  of $\xi$, by
 	\begin{equation*}
 		\underline{\xi}_t = \inf_{0\leq s\leq t} \xi_s \qquad \textrm{ and } \qquad \overline{\xi}_t = \sup_{0 \leq s \leq t} \xi_s, \qquad \textrm{for} \qquad t \geq 0.
 	\end{equation*}
 	For our purposes, we also introduce and provide some useful properties of the  L\'evy process reflected at their running infimum and supremum. Recall that the 
 	reflected processes $\xi-\underline{\xi}$ and  $\overline{\xi}-\xi$  are Markov processes with respect to the  filtration $(\mathcal{F}^{(e)}_t)_{t\geq 0}$ and whose semigroups
 	satisfy the Feller property (see for instance Proposition VI.1 in the monograph of Bertoin \cite{bertoin1996levy}).  We 
 	denote by $L=(L_t, t \geq 0 )$ and $\widehat{L}=(\widehat{L}_t, t \geq 0 )$  for  the local times of $\overline{\xi}-\xi$ and $\xi-\underline{\xi}$ at $0$, respectively,  in the sense of Chapter IV in \cite{bertoin1996levy}. If $0$ is regular for $(-\infty,0)$ or regular downwards, i.e.
 	\[
 	\mathbb{P}^{(e)}(\tau^-_0=0)=1,
 	\] 
 	where $\tau^{-}_0=\inf\{s> 0:  \xi_s\le 0\}$, then $0$ is regular for the reflected process $\xi-\underline{\xi}$ and then, up to a multiplicative constant, $\widehat{L}$ is the unique additive functional of the reflected process whose set of increasing points is $\{t:\xi_t=\underline{\xi}_t\}$. If $0$ is not regular downwards then the set $\{t:\xi_t=\underline{\xi}_t\}$ is discrete and we define the local time $\widehat{L}$ as the counting process of this set. The same properties holds for $L$ by duality, i.e. if $0$ is regular upwards then, up to a multiplicative constant, $L$ is the unique additive functional  whose set of increasing points is $\{t:\xi_t=\overline{\xi}_t\}$, otherwise $L$ is the counting process of this set.
 	
 Let us denote by $L^{-1}$ and $\widehat{L}^{-1}$ the right continuous inverse of the local times $L$ and $\widehat{L}$, respectively.  The range of the inverse local times  $L^{-1}$ and $\widehat{L}^{-1}$, correspond to the sets of real times at which new maxima and new minima occur, respectively. Next, we introduce the so called increasing ladder height process by
 	\begin{equation}\label{defwidehatH}
 		H_t=\overline{\xi}_{L_t^{-1}}, \qquad t\ge 0.
 	\end{equation}
 	The pair $(L^{-1}, H)$ is a bivariate subordinator, as is the case of  the pair $(\widehat{L}^{-1}, \widehat{H})$ with
 	\[
 	\widehat{H}_t=-\underline{\xi}_{\widehat{L}_t^{-1}}, \qquad t\ge 0.
 	\]

 	Similarly to the critical case, which was studied by  Bansaye et al. \cite{bansaye2021extinction}, the asymptotic analysis  of the intermediate subcritical regime requires the notion of the renewal functions $U^{(\theta)}$ and $\widehat{U}^{(\theta)}$ under $\mathbb{P}^{(e,\theta)}$, which are associated to the supremum and infimum of $\xi$, respectively. More precisely,   , for all $x>0$, we introduce the renewal functions $U^{(\theta)}$ and $\widehat{U}^{(\theta)}$ as follows
 	\begin{equation}\label{eq_Utheta}
 		U^{(\theta)}(x) := \mathbb{E}^{(e,\theta)}\left[\int_{[0,\infty)} \mathbf{1}_{\left\{\overline{\xi}_t\leq x\right\}} \mathrm{d} L_t\right] \quad 
 		\textrm{and}\quad
 		\widehat{U}^{(\theta)}(x) := \mathbb{E}^{(e,\theta)}\left[\int_{[0,\infty)} \mathbf{1}_{\left\{\underline{\xi}_t\geq -x\right\}} \mathrm{d} \widehat{L}_t\right].
 	\end{equation}
The renewal functions are identically 0 on $(-\infty, 0]$,  strictly positive on $(0,\infty)$  and satisfy 
 	\begin{equation}
 		\label{grandO}
 		U^{(\theta)}(x)\leq C_1 x \qquad\textrm{and}\qquad \widehat{U}^{(\theta)}(x)\leq C_2 x \quad \text{  for any } \quad x\geq 0,
 	\end{equation}
 	where $C_1, C_2$ are finite constants  (see for instance  Lemma 6.4 and Section 8.2 in the monograph of Doney 
 	\cite{doney2007fluctuation}). 
%

\subsection{Strongly subcritical regime}

\begin{proof}[Proof of Theorem \ref{theo_stronglycase}]
	Let $z>0$ and $x\in \mathbb{R}$; and denote by $\mathbb{P}_{(z,x)}$ for the law of the couple $(Z, \xi)$ starting from $z$ and $x$, respectively. 	We begin by noting that, conditioning on the environment and then  using the exponential change of measure given in \eqref{eq_medida_Essher} with $\theta=1$, allow us to deduce
	\[\begin{split}
		e^{-\Phi_{\xi}(1)t}\mathbb{P}_{z}(Z_t>0) &=  e^{-\Phi_{\xi}(1)t} \mathbb{E}^{(e)}\Big[e^{-\xi_t}e^{\xi_t}\mathbb{P}_{(z,0)}\big(Z_t>0\ \big|\big. \ \xi\big)\Big] \\ &= \mathbb{E}^{(e,1)}\Big[e^{-\xi_t}\mathbb{P}_{(z,0)}\big(Z_t>0\ \big|\big. \ \xi\big)\Big]. 
	\end{split}\]
	Recall from  \eqref{eq_transLaplace}, that for $\lambda\geq 0$ and $t\geq 0$, the random cumulant  $h_{0,t}(\lambda)=e^{-\xi_0}v_t(s,\lambda e^{\xi_t},\xi)$ satisfies 
	$$ \mathbb{E}_{(z,0)}\big[e^{-\lambda Z_t}\  \big|\big. \  \xi\big]=\exp\{-zh_{0,t}(\lambda)\}.$$
	Now, we denote 
	\begin{equation*}
		G_t(\lambda) := e^{-\xi_t}\Big(1-\exp\{-zh_{0,t}(\lambda)\}\Big), \quad \quad \text{for}\quad \lambda, t\geq 0,
	\end{equation*}
and observe that the quenched survival probability of  $Z$ is given by 
\[
\mathbb{P}_{(z,0)}\big(Z_t>0 \  \big|\big. \  \xi\big)=1-\exp\{-zh_{0,t}(\infty)\}.
\]
 In other words, 
	\[
	G_t(0)= 0 \quad \quad \quad  \text{and}\quad \quad  \quad G_t(\infty) = e^{-\xi_t}\mathbb{P}_{(z,0)}\big(Z_t>0 \  \big|\big. \  \xi\big).
	\]
	Since the map $\lambda \mapsto h_{0,t}(\lambda)$ is differentiable, then so does $G_t(\cdot)$.
	In view of the above arguments, we deduce  
	\begin{equation}\label{eq_limZ}
		e^{-\Phi_{\xi}(1)t}\mathbb{P}_{z}(Z_t>0) =  \mathbb{E}^{(e,1)}\big[G_t(\infty)\big] = \mathbb{E}^{(e,1)}\left[\int_{0}^{\infty} G_t'(\lambda)\mathrm{d}\lambda\right]  =  \int_{0}^{\infty} \mathbb{E}^{(e,1)}\big[G_t'(\lambda)\big]\mathrm{d} \lambda,
	\end{equation}
	where in the last equality, the expectation and the integral may be exchanged thanks to Fubini's Theorem. 
	Hence in order to deduce our result, we would like to take the limit, as $t\to \infty$, in the above equality and then make use of  the Dominated Convergence Theorem in order to interchange the  limit with the integral on the right-hand side.
	With this purpose in mind,  we need to find a function $g(\lambda)$ such that $\mathbb{E}^{(e,1)}\big[|G'_t(\lambda)|\big]\leq g(\lambda), \ \text{for all} \ t\geq 1 $, and  
	\begin{equation}\label{eq_ht}
		\int_{0}^{\infty} g(\lambda) \mathrm{d} \lambda <\infty.
	\end{equation}
	First, we analyse  $\mathbb{E}^{(e,1)}\big[|G'_t(\lambda)|\big]$. Note  from the definition of $G_t(\lambda)$ that
	\begin{equation}\label{eq:derivada}
		G_t'(\lambda) = ze^{-\xi_t}\exp\big\{-zh_{0,t}(\lambda)\big\} h'_{0,t}(\lambda)=z\exp\big\{-zh_{0,t}(\lambda)\big\} \frac{\mathrm{d} }{\mathrm{d}u}v_t(0, u,\xi)\Big|\Big._{u=\lambda e^{\xi_t}},
	\end{equation}
	where in the last equality we  recall that $h_{0,t}(\lambda)=e^{-\xi_0}v_t(0,\lambda e^{\xi_t},\xi)$.
	Moreover, by differentiating with respect to $\lambda$ on both sides of the backward differential equation \eqref{eq_BDE}, we obtain
	\begin{equation*}
		\frac{\mathrm{d}}{\mathrm{d} \lambda} v_t(0,\lambda, \xi) = 1-\int_{0}^{t} \psi'_0\big(e^{-\xi_s}v_t(s,\lambda,\xi)\big)\frac{\mathrm{d}}{\mathrm{d} \lambda} v_t(s,\lambda,\xi)\mathrm{d} s.
	\end{equation*}
	Thus solving the above equation, we get
	\begin{equation}\label{eq:derivatev}
		\frac{\mathrm{d}}{\mathrm{d} \lambda} v_t(0,\lambda, \xi) = \exp\left\{-\int_{0}^{t} \psi'_0\big(e^{-\xi_s}v_t(s,\lambda,\xi)\big)\mathrm{d} s\right\}.
	\end{equation}
	Then, it follows
	\begin{equation*}
		\mathbb{E}^{(e,1)}\big[|G_t'(\lambda)|\big]= \mathbb{E}^{(e,1)}\big[G_t'(\lambda)\big]=  z\mathbb{E}^{(e,1)}\left[\exp\left\{-zh_{0,t}(\lambda)-\int_{0}^{t} \psi'_0\big(h_{s,t}(\lambda)\big)\mathrm{d} s\right\}\right].
	\end{equation*}
	In other words, according to identity \eqref{eq:laplaceimigration}, $G_t'(\lambda)$ is the Laplace transform of a CBLE with immigration and whose  immigration mechanism  is given by $\psi_0'$. 
	
	In order to find the  integrable function $g(\lambda)$ which dominates $\mathbb{E}^{(e,1)}[|G'_t(\lambda)|],$ for $ t\geq 1$, we use another useful characterisation of $\mathbb{E}^{(e,1)}[G_t'(\lambda)]$.  Recall  that the homogeneous L\'evy process $\Xi$  defined in \eqref{eq_levytiemponeg}, allows to extend the definition of the map $s\mapsto h_{s,0}^{\Xi}(\lambda)$ for $s\leq 0$. The latter is the unique positive  pathwise solution to \eqref{eq_BDEhneg}. We write, for $\lambda>0$ and $t\geq 0$,

	
	\begin{equation*}
		\mathbb{E}^{(e,1)}\big[G_t'(\lambda)\big]=z\mathbf{E}^{(e,1)}\left[\exp\left\{-zh_{-t,0}^{\Xi}(\lambda)-\int_{-t}^{0} \psi'_0\big(h_{s,0}^{\Xi}(\lambda)\big)\mathrm{d} s\right\}\right],
	\end{equation*}
where  $\mathbf{P}^{(e,1)}$ denotes the law of the homogeneous L\'evy process $\Xi$ constructed in \eqref{eq_levytiemponeg} but with $(\xi,\mathbb{P}^{(e,1)})$. Now, we introduce the function
	\begin{equation*}
		g(\lambda) := \mathbf{E}^{(e,1)}\left[\exp\left\{-\int_{-1}^{0} \psi'_0\big(h_{s,0}^{\Xi}(\lambda)\big)\mathrm{d} s\right\}\right].
	\end{equation*}
	Using the latter characterisation of  $\mathbb{E}^{(e,1)}[G_t'(\lambda)]$ together with the non-negative property of $\psi'_0$ and $h_{-t,0}^{\Xi}(\lambda)$, we deduce the following inequality, for $t\geq 1$,
	\[
	\begin{split}
		\mathbb{E}^{(e,1)}\big[|G_t'(\lambda)|\big] &\leq z \mathbf{E}^{(e,1)}\left[\exp\left\{-\int_{-t}^{0} \psi'_0\big(h_{s,0}^{\Xi}(\lambda)\big)\mathrm{d} s\right\}\right]\leq  zg(\lambda).
	\end{split}\]
	Furthermore, observe from  assumption \eqref{eq:hipstrongcase}, that the function $g(\cdot)$ is  integrable. Thus, we appeal to the Dominated Convergence Theorem in \eqref{eq_limZ} and get
	\[\begin{split}
		\lim\limits_{t\to\infty}e^{-\Phi_{\xi}(1)t}\mathbb{P}_{z}(Z_t>0)  & = \lim\limits_{t\to\infty} \int_{0}^{\infty} \mathbb{E}^{(e,1)}\big[G_t'(\lambda)\big]\mathrm{d} \lambda \\ &=  \int_{0}^{\infty} \lim\limits_{t\to\infty} \mathbb{E}^{(e,1)}\big[G_t'(\lambda)\big]\mathrm{d} \lambda = \int_{0}^{\infty}  \mathbb{E}^{(e,1)} \left[\lim\limits_{t\to\infty}G_t'(\lambda)\right]\mathrm{d} \lambda,
	\end{split}\]
	where we have used again the Dominated Convergence Theorem in the last equality since the inequality $|G_t'(\lambda)|\leq z$ holds for all $t\geq 1$. 
	
	On the other hand, we also note that assumption  $\Phi_\xi'(1)<0$ implies $\xi_t\to -\infty$ as $t\to\infty$, $\mathbb{P}^{(e,1)}$-a.s. The latter implies that \ $\Xi_t\to \infty$\  as  $t\to-\infty$, $\mathbf{P}^{(e,1)}$-a.s. Next, thanks to the monotonicity property (see Proposition 2.3 in \cite{he2018continuous}) of  the map $-t \mapsto v_0(-t,\lambda,\Xi)$, we have $$h_{-t,0}^{\Xi}(\lambda)=e^{-\Xi_{-t}}v_{0}(-t,\lambda,\Xi)\leq e^{-\Xi_{-t}}v_{0}(0,\lambda,\Xi)=e^{-\Xi_{-t}}\lambda.$$  It then follows that $\lim_{t\to\infty}h_{-t,0}^{\Xi}(\lambda)=0$, $\mathbf{P}^{(e,1)}$-a.s., and thus
	\[\begin{split}
		\mathbb{E}^{(e,1)}\left[\lim\limits_{t\to\infty}G_t'(\lambda)\right] &=z \mathbf{E}^{(e,1)}\left[\lim\limits_{t\to\infty}\exp\left\{-zh_{-t,0}^{\Xi}(\lambda)-\int_{-t}^{0}\psi_0'\big(h_{s,0}^{\Xi}(\lambda)\big)\mathrm{d} s\right\}\right] \\ 
		& = z\mathbf{E}^{(e,1)}\left[\exp\left\{-\int_{-\infty}^{0}\psi_0'\big(h_{s,0}^{\Xi}(\lambda)\big)\mathrm{d} s\right\}\right].
	\end{split}
	\]
	The proof is complete once we have shown that
	$$0<\mathfrak{B}_1:=\int_{0}^{\infty}  \mathbf{E}^{(e,1)}\left[\exp\left\{-\int_{-\infty}^{0}\psi_0'\big(h_{s,0}^{\Xi}(\lambda)\big)\mathrm{d} s\right\}\right]\mathrm{d} \lambda <\infty.$$
	From the discussion at the end of subsection \ref{sec_defandprop} (see also Corollary 5.7 in He et al. \cite{he2018continuous} or the proof of  Theorem 5.6 in the same reference), we see that under the moment condition  \eqref{eq_xlogx}, we have
	\begin{equation*}
		\mathbf{E}^{(e,1)}\left[\exp\left\{-\int_{-\infty}^{0}\psi_0'\big(h_{s,0}^{\Xi}(\lambda)\big)\mathrm{d} s\right\}\right]>0.
	\end{equation*}
	Therefore 
	\[\begin{split}
		\int_{0}^{\infty}  \mathbb{E}^{(e,1)}\left[\lim\limits_{t\to\infty}G_t'(\lambda)\right]\mathrm{d} \lambda &=z\int_{0}^{\infty}\mathbf{E}^{(e,1)}\left[\exp\left\{-\int_{-\infty}^{0}\psi_0'\big(h_{s,0}^{\Xi}(\lambda)\big)\mathrm{d} s\right\}\right]\mathrm{d} \lambda \\
		& =z\mathfrak{B}_1>0.
	\end{split}\]
	Finally, since $\psi'_0$ is non-negative and the condition in \eqref{eq:hipstrongcase}, we obtain the finiteness of $\mathfrak{B}_1$, that is   
	$$ \mathfrak{B}_1 \leq \int_{0}^{\infty} \mathbf{E}^{(e,1)}\left[\exp\left\{-\int_{-1}^{0}\psi_0'\big(h_{s,0}^{\Xi}(\lambda)\big)\mathrm{d} s\right\}\right] \mathrm{d} \lambda =\int_{0}^{\infty} g(\lambda) \mathrm{d} \lambda < \infty. $$
	This completes the proof. 
	
\end{proof}

\subsubsection{The stable case}\label{sec:stablecase}
Now, let us  compute the constant  $\mathfrak{B}_1$ in the stable case and check that it coincides with the constant that appears  in Theorem 5.1 in Li and Xu \cite{li2018asymptotic}. To this end,  we recall the following identity of the previous proof  
\[z \mathbf{E}^{(e,1)}\left[\exp\left\{-\int_{-\infty}^{0}\psi_0'(h_{s,0}^{\Xi}(\lambda))\mathrm{d} s\right\}\right] = \lim\limits_{t\to \infty}\mathbb{E}^{(e,1)}\left[ G_t'(\lambda)\right],\]
where $G_t'(\lambda)$ is given  in \eqref{eq:derivada}. We also recall that in  this case  $\psi_0(\lambda)= C\lambda^{1+\beta}$ with $\beta\in (0,1)$ and $C>0$, from which  we observe that the $x\log x$-moment condition \eqref{eq_xlogx} and Grey's condition \eqref{eq:grey} are clearly satisfied. Moreover the  backward differential equation in \eqref{eq_BDE} can be solved explicitly (see e.g. Section 5 in \cite{he2018continuous}), that is
$$ v_t(s,\lambda, \xi) = \Big(\lambda^{-\beta} + \beta C\texttt{I}_{s,t}(\beta \xi)\Big)^{-1/\beta},$$
where $\texttt{I}_{s,t}(\beta \xi)$  denotes the exponential functional of  the L\'evy process $\beta \xi$, i.e., 
\begin{equation}\label{eq_expfuncLevy}
	\texttt{I}_{s,t}(\beta \xi):=\int_{s}^{t}e^{-\beta \xi_u} \mathrm{d} u, \quad \quad  0\leq s\leq t. 
\end{equation}
In other words, we obtain 
\[\begin{split}
	\mathbf{E}^{(e,1)}&\left[\exp\left\{-\int_{-\infty}^{0}\psi_0'(h_{s,0}^{\Xi}(\lambda))\mathrm{d} s\right\}\right] \\ 
	&\hspace{1cm}= \lim\limits_{t\to\infty}\mathbb{E}^{(e,1)}\left[ \exp\{-zv_t(0,\lambda e^{\xi_t},\xi)\}\Big(1+(\lambda e^{\xi_t})^\beta \beta C \texttt{I}_{0,t}(\beta \xi)\Big)^{-\frac{1}{\beta}-1}\right].
\end{split}\]
Now appealing to  duality, see  \eqref{eq_lemmaduality}, we get
\begin{equation*}
	e^{\beta \xi_t}\texttt{I}_{0,t}(\beta \xi) = \int_{0}^{t} e^{-\beta (\xi_u-\xi_t)} \mathrm{d} u \stackrel{(d)}{=} \int_{0}^{t} e^{\beta \xi_u} \mathrm{d} u =\texttt{I}_{0,t}(-\beta \xi),
\end{equation*}
and 
\[\begin{split}
	v_t(0,\lambda e^{\xi_t}, \xi) = e^{\xi_t}\left(\lambda^{-\beta} + \beta Ce^{\beta \xi_t}\texttt{I}_{0,t}(\beta \xi)\right)^{-1/\beta}  \stackrel{(d)}{=} e^{\xi_t}\left(\lambda^{-\beta} + \beta C\texttt{I}_{0,t}(-\beta \xi)\right)^{-1/\beta} .
\end{split}\]
Hence
\[\begin{split}
	\mathbf{E}^{(e, 1)}&\left[\exp\left\{-\int_{-\infty}^{0}\psi_0'(h_{s,0}^{\Xi}(\lambda))\mathrm{d} s\right\}\right] \\ &= \lim\limits_{t\to\infty}\mathbb{E}^{(e,1)}\left[ \exp\left\{-ze^{\xi_t}\left(\lambda^{-\beta} + \beta C\texttt{I}_{0,t}(-\beta \xi)\right)^{-1/\beta}\right\}\Big(1+\lambda^\beta \beta C \texttt{I}_{0,t}(-\beta \xi)\Big)^{-\frac{1}{\beta}-1}\right].
\end{split}\]
Furthermore since $\xi_t \to -\infty$, as $t\to \infty$, $\mathbb{P}^{(e,1)}$-a.s., then $\mathtt{I}_{0,\infty}(-\beta\xi)$ is finite $\mathbb{P}^{(e,1)}$-a.s.  Thus, it follows that  $$\lim\limits_{t\to\infty}\exp\left\{-ze^{\xi_t}\left(\lambda^{-\beta} + \beta C\texttt{I}_{0,t}(-\beta \xi)\right)^{-1/\beta}\right\}=1, \qquad \mathbb{P}^{(e,1)}-\text{a.s.},$$ 
which yields,
\[\begin{split} \mathbf{E}^{(e,1)}\left[\exp\left\{-\int_{-\infty}^{0}\psi_0'(h_{s,0}^{\Xi}(\lambda))\mathrm{d} s\right\}\right]  =  \mathbb{E}^{(e,1)}\left[\Big(1+\beta C \lambda^{\beta} \texttt{I}_{0,\infty}(-\beta \xi)\Big)^{-\frac{1}{\beta}-1}\right]\end{split}.\]
Next, we claim that condition \eqref{eq:hipstrongcase} is satisfied  under assumption \eqref{eq_moments} with $\vartheta=1$. We prove this claim below.  Hence, using Fubini\'s Theorem we deduce
\[\begin{split} \mathfrak{B}_1&=  \int_{0}^{\infty} \mathbb{E}^{(e,1)}\left[\left(1+\beta C \lambda^{\beta} \texttt{I}_{0,\infty}(-\beta \xi)\right)^{-\frac{1}{\beta}-1}\right]\mathrm{d} \lambda \\ &= (\beta C)^{-1/\beta} \mathbb{E}^{(e,1)}\left[\left(\int_{0}^{\infty}e^{\beta \xi_u}\mathrm{d} u\right)^{-1/\beta}\right], \end{split}\]
where in the last equality  we have solved the integral with respect to $\lambda$.

Finally, we prove the claim that  condition \eqref{eq:hipstrongcase} is satisfied under Assumption \ref{eq_moments} with $\vartheta=1$. We first observe that the integral in condition \eqref{eq:hipstrongcase} can be rewritten as follows
\begin{equation*}
	\begin{split}
		\int_{0}^{\infty}\mathbf{E}^{(e,1)}\left[\exp\left\{-\int_{-1}^{0} \psi'_0\big(h^{\Xi}_{s,0}(\lambda)\big)\mathrm{d} s\right\}\right] \mathrm{d}\lambda& = \int_{0}^{\infty} \mathbb{E}^{(e,1)}\left[\Big(1+\lambda^\beta \beta C \texttt{I}_{0,1}(-\beta \xi)\Big)^{-\frac{1}{\beta}-1}\right] \mathrm{d} \lambda,\\& = \mathbb{E}^{(e,1)}\left[\left(\int_{0}^{1}e^{\beta \xi_u}\mathrm{d} u\right)^{-1/\beta}\right],
	\end{split}
\end{equation*}
where  in the last equality we have used again Fubini's Theorem. In particular, appealing once again to  duality \eqref{eq_lemmaduality}, we get
\begin{equation*}
	\begin{split}
		 \mathbb{E}^{(e,1)}\left[\left(\int_{0}^{1}e^{\beta \xi_u}\mathrm{d} u\right)^{-1/\beta}\right] = e^{-\Phi_{\xi}(1)}\mathbb{E}^{(e)}\left[\left(\int_{0}^{1}e^{-\beta\xi_u}\mathrm{d} u\right)^{-1/\beta}\right] = 	 e^{-\Phi_{\xi}(1)}\mathbb{E}^{(e)}\big[\mathtt{I}_{0,1}(\beta \xi)^{-1/\beta}\big].
	\end{split}
\end{equation*}
Moreover, under the exponential moment Assumption \ref{eq_moments} with $\vartheta=1$, from Lemma 2.2 in \cite{he2018continuous}, we deduce that
\begin{equation*}
	\mathbb{E}^{(e)}\Big[\mathtt{I}_{0,1}(\beta \xi)^{-1/\beta}\Big] \leq e^{2\Phi_\xi'(0)}\mathbb{E}^{(e)}\big[e^{\xi_1}\big],
\end{equation*} 
where the right-hand side is finite from our assumption. This prove our claim.

\subsection{Intermediate subcritical regime}\label{sec:intermediate}
The aim of this section is to show Theorem \ref{theo_interm}. Throughout this section, we assume that the underlying L\'evy process $\xi$ fulfils  conditions $\Phi_\xi'(0)<0$ and  $\Phi_\xi'(1)=0$. In other words, $\xi$ drifts to $-\infty$ under $\mathbb{P}^{(e)}$ and oscillates under  the probability measure  $\mathbb{P}^{(e,1)}$ defined by the Esscher transform \eqref{eq_Utheta}.

Before moving to the proof  of  Theorem  \ref{theo_interm}, we recall that, under the assumption that the L\'evy process $\xi$ possesses exponential moments of order $\vartheta>1$, the probability that the supremum of  $\xi$ stays below $0$ under $\mathbb{P}^{(e,1)}_{x}$, for $x<0$, satisfies
\begin{equation} \label{eq_int_cota_max}
	\mathbb{P}^{(e,1)}_{x}\big(\overline{\xi}_{t}< 0\big) \sim \sqrt{\frac{2}{\pi \Phi_\xi''(1)}} \mathbb{E}^{(e,1)}\big[H_1\big]U^{(1)}(-x)t^{-1/2},\ \qquad \text{as}\quad  t\to \infty,
\end{equation}
where we recall that $U^{(1)}$ denotes the renewal function, under $\mathbb{P}^{(e,1)}$, and $(H_t, t\geq 0)$ the ascending ladder process, (see Lemma 11 in Hirano \cite{hirano2001levy}). 

The proof of Theorem \ref{theo_interm} uses the same notation as in the proof of Theorem \ref{theo_stronglycase} and is based on the following two lemmas. The first of which
 tell us, under our general assumptions \eqref{eq_psiupperbound} and the exponential moments condition \eqref{eq_moments} with $\vartheta>1$, that only paths of L\'evy processes with a low supremum contribute to the probability of survival.
\begin{lemma}\label{lem_inter_cota0}
	Suppose that condition \eqref{eq_moments} holds with $\vartheta>1$.  We also assume that condition  \eqref{eq_psiupperbound} is satisfied.  Then for any $z> 0$, $x<0$ and $0<\delta<1$, we have 
	\begin{equation*}
		\lim\limits_{y \to \infty} \limsup_{t \to \infty} t^{1/2} e^{-\Phi_\xi(1)t} \mathbb{P}_{(z,x)}\Big(Z_t>0, \ \overline{\xi}_{t-\delta}\geq y\Big) =0.
	\end{equation*}
\end{lemma}

\begin{proof}
	Let $z >0$, $x<0$ and $0<\delta <1$. We begin by noting that conditioning on $\xi$ and  then using the Esscher transform allow us to deduce that
	\begin{eqnarray*}
		e^{-t\Phi_\xi(1)}\mathbb{P}_{(z,x)}\Big(Z_t>0, \ \overline{\xi}_{t-\delta}\geq y\Big) &=&  \mathbb{E}_x^{(e,1)}\Big[e^{-\xi_t}\mathbb{P}_{(z,x)}\big(Z_t>0 \  \big|\big. \ \xi \big)\mathbf{1}_{\{ \overline{\xi}_{t-\delta}\geq y\}}\Big].
	\end{eqnarray*}
	Further note from \eqref{eq_Laplace}, that the survival probability conditioned on the environment is bounded from above by the functional $v_t(0,\infty, \xi-\xi_0)$, i.e., 
	\begin{eqnarray*}\label{eq_strongly_cota}
		\mathbb{P}_{(z,x)}\big(Z_t>0\  \big|\big. \  \xi\big)&=&1-\exp\big\{-zv_t(0,\infty, \xi-\xi_0)\big\} \nonumber \leq z v_t(0,\infty, \xi-\xi_0).
	\end{eqnarray*}
	
	On the other hand,  condition \eqref{eq_psiupperbound} allows us to find a lower bound for $v_t(0,\infty,\xi-\xi_0)$ in terms of the exponential functional of $\xi$. Indeed, we observe from the backward differential equation given in \eqref{eq_BDE} that 
	\begin{equation*}
		\frac{\partial }{\partial s} v_t(s,\lambda e^{-\xi_0},\xi-\xi_0) \geq C v_t^{1+\beta}(s,\lambda  e^{-\xi_0},\xi)e^{-\beta (\xi_s-\xi_0)}, \quad v_t(t,\lambda e^{-\xi_0}, \xi-\xi_0)=\lambda e^{-\xi_0}.
	\end{equation*}
	Integrating between 0 and $t$, we get
	\begin{equation*}
		\frac{1}{v_t^\beta(0,\lambda e^{-\xi_0},\xi-\xi_0)} -\frac{1}{(\lambda e^{-\xi_0})^\beta} \geq C\beta \int_{0}^{t} e^{-\beta (\xi_s-\xi_0)}\mathrm{d} s\quad\text{with} \quad C\beta >0.
	\end{equation*}
	Now, letting $\lambda\uparrow \infty$ and taking into account that  $\beta\in(0,1)$ and $C>0$, we deduce the following inequality  for all $t\geq 0$,
	\begin{equation}\label{eq_cota_v_I}
		v_t(0,\infty,\xi-\xi_0)\leq \big(C\beta \texttt{I}_{0,t}(\beta (\xi-\xi_0))  \big)^{-1/\beta},
	\end{equation}
	where $\texttt{I}_{0,t}(\beta (\xi-\xi_0))$ is the exponential functional of  the L\'evy process $\beta (\xi-\xi_0)$, see  \eqref{eq_expfuncLevy}. The latter implies that
	\[\begin{split}
		e^{-t\Phi_\xi(1)}\mathbb{P}_{(z,x)}\Big(Z_t>0, \ \overline{\xi}_{t-\delta}\geq y\Big) &\leq z(\beta C)^{-1/\beta} \mathbb{E}_x^{(e,1)}\Big[e^{-\xi_t}  \texttt{I}_{0,t}(\beta (\xi-\xi_0))^{-1/\beta}\mathbf{1}_{\{ \overline{\xi}_{t-\delta}\geq y\}}\Big] \\ &= z(\beta C)^{-1/\beta} \mathbb{E}^{(e,1)}\Big[ \texttt{I}_{0,t}(-\beta \xi)^{-1/\beta}\mathbf{1}_{\{ \underline{\xi}_{t-\delta}\leq -y-x\}}\Big],
	\end{split}\]
	where in the last equality we have appealed to the duality Lemma given in \eqref{eq_lemmaduality} to see that
	\begin{equation*}
		e^{-\xi_t}\texttt{I}_{0,t}(\beta \xi)^{-1/\beta} \stackrel{(d)}{=} \left(\int_{0}^{t} e^{\beta \xi_s} \mathrm{d} s\right)^{-1/\beta} =\texttt{I}_{0,t}(-\beta \xi) ^{-1/\beta}.
	\end{equation*}
	Finally, according to Li and Xu \cite[Lemma 3.5]{li2018asymptotic}, we have 
	$$ \lim_{y\to\infty} \limsup_{t\to \infty}	t^{1/2}  \mathbb{E}^{(e,1)}\Big[\texttt{I}_{0,t}(-\beta \xi)^{-1/\beta} \mathbf{1}_{\{ \underline{\xi}_{t-\delta}\leq  -y-x\}}\Big] =0.$$
	Therefore,
	\[\begin{split}
		\lim\limits_{y\to\infty}\limsup_{t\to \infty} &	t^{1/2}  e^{-\Phi_\xi(1)t} \mathbb{P}_{(z,x)}\Big(Z_t>0, \ \overline{\xi}_{t-\delta}\geq y\Big) \\ &\leq z(\beta C)^{-1/\beta}\lim_{y\to\infty} \limsup_{t\to \infty}t^{1/2}  \mathbb{E}^{(e,1)}\Big[\texttt{I}_{0,t}(-\beta \xi)^{-1/\beta} \mathbf{1}_{\{ \underline{\xi}_{t-\delta}\leq  -y-x\}}\Big] =0,
	\end{split}\]
	which concludes the proof.
\end{proof}

The following Lemma studies the survival probability under environments with low extrema. More precisely, it confirms the statement that only paths of the L\'evy process with a very low running supremum give a substantial contribution to the speed of the survival probability.

\begin{lemma}\label{lem_inter_cota1} Suppose that condition \eqref{eq_xlogx} holds together with the exponential moment condition \eqref{eq_moments}  with  $\vartheta>1$. We also assume that the integral condition in \eqref{eq:assumpinter} holds Then for every $z>0$ and $x<0$, we have 
	\[\begin{split}
		\lim\limits_{t\to\infty} 	t^{1/2} e^{-\Phi_{\xi}(1)t}\mathbb{P}_{(z,x)}\Big(Z_t>0, \ \overline{\xi}_t< 0\Big)    = z \sqrt{\frac{2}{\pi \Phi_\xi''(1)}} \mathbb{E}^{(e,1)}\big[H_1\big]\mathfrak{b}_2(x),
	\end{split}\]
	where 
	\begin{equation}\label{eq_B3}
		\mathfrak{b}_2(x)= U^{(1)}(-x) \mathbf{E}^{(e,1), \uparrow}_{-x}\left[\int_{0}^{\infty}\exp\left\{-\int_{-\infty}^{0}\psi_0'\big(h_{s,0}^{\Xi}(\lambda)\big)\mathrm{d} s\right\}\mathrm{d} \lambda\right] \in (0,\infty).
	\end{equation}
\end{lemma}

\begin{proof}
	Let $z>0$ and assume that $\xi_0=x<0$. We begin by recalling that,  under $\mathbb{P}^{(e,1)}$, the L\'evy process $\xi$ oscillates. In addition from the Esscher transform,
	we have the following identity 
	\begin{equation*}
		\begin{split}
			e^{-\Phi_{\xi}(1)t}\mathbb{P}_{(z,x)}\Big(Z_t>0, \ \overline{\xi}_t < 0\Big) &= 	e^{-\Phi_{\xi}(1)t}\mathbb{P}_{(z,0)}\Big(Z_t>0, \ \overline{\xi}_t < -x\Big) \\ &  =\mathbb{E}^{(e,1)}\Big[e^{-\xi_t}\mathbb{P}_{(z,0)}\big(Z_t>0\ \big|\big. \ \xi\big)\mathbf{1}_{\{\overline{\xi}_t < -x\}}\Big] 
		\end{split}
	\end{equation*}
	Recall from  \eqref{eq_transLaplace}, that for any $\lambda\geq 0$ and $s\leq t$, the random cumulant  $h_{s,t}(\lambda)$ satisfies 
	\begin{equation*}
		\begin{split}
			\mathbb{E}_{(z,x)}\big[e^{-\lambda Z_t}\  \big|\big. \  \xi, \mathcal{F}_s^{(b)}\big]&= 	\mathbb{E}_{(z,0)}\Big[e^{-\lambda Z_t e^{\xi_t} e^{-\xi_t}}\  \big|\big. \  \xi, \mathcal{F}_s^{(b)}\Big] \\ &=\exp\{-Z_s h_{s,t}(\lambda)\}.
		\end{split}
	\end{equation*}
	From the previous identity, we observe that the initial condition of the L\'evy process $\xi$ is irrelevant for the functional $h_{s,t}(\lambda)$. Further, recall that the quenched survival probability of the process $(Z_t,t\geq0)$ is given by $\mathbb{P}_{(z,0)}\big(Z_t>0 \  \big|\big. \  \xi\big)=1-\exp\{-zh_{0,t}(\infty)\}$. Thus,
	\begin{equation*}
		\begin{split}
			e^{-\Phi_{\xi}(1)t}\mathbb{P}_{(z,x)}\Big(Z_t>0, \ \overline{\xi}_t < 0\Big)   &= \mathbb{E}^{(e,1)}\Big[e^{-\xi_t}\mathbb{P}_{(z,0)}\big(Z_t>0\ \big|\big. \ \xi\big)\mathbf{1}_{\{\overline{\xi}_t < -x\}}\Big]\\ &= \mathbb{E}^{(e,1)}\Big[e^{-\xi_t} \big(1-\exp\big\{-zh_{0,t}(\infty)\big\}\big)\mathbf{1}_{\{\overline{\xi}_t < -x\}}\Big]. 
		\end{split}
	\end{equation*}
	Now, we use the same notation as in the proof of Theorem \ref{theo_stronglycase}. Namely, we denote for each fixed $t\geq 0$, the function
	\begin{equation*}
		G_t(\lambda) = e^{-\xi_t}\Big(1-\exp\big\{-zh_{0,t}(\lambda)\big\}\Big), \quad \quad \text{for}\quad \lambda\geq 0.
	\end{equation*}
	Then,
	\[
	G_t(0)= 0 \quad \quad \quad  \text{and}\quad \quad  \quad G_t(\infty) = e^{-\xi_t}\mathbb{P}_{(z,0)}\big(Z_t>0 \  |\  \xi\big).
	\]
	Since the map $\lambda \mapsto h_{0,t}(\lambda)$ is differentiable, then so does $G_t(\cdot)$.
	In view of the above arguments, we deduce  
	\begin{eqnarray*}
		e^{-\Phi_{\xi}(1)t}\mathbb{P}_{(z,x)}\Big(Z_t>0, \ \overline{\xi}_t < 0\Big) &=& \mathbb{E}^{(e,1)}\left[\mathbf{1}_{\{\overline{\xi}_t< -x\}}\int_{0}^{\infty} G_t'(\lambda)\mathrm{d}\lambda \right]\\ & =& \int_{0}^{\infty} \mathbb{E}^{(e,1)}\Big[G_t'(\lambda) \mathbf{1}_{\{\overline{\xi}_t< -x\}}\Big]\mathrm{d} \lambda,
	\end{eqnarray*}
	where in the last equality, the expectation and the integral may be exchanged using Fubini's Theorem. Recall the definition of the homogeneous L\'evy process $\Xi$ given in \eqref{eq_levytiemponeg}. Now, using the same strategy as in the proof of Theorem \ref{theo_stronglycase}, that is extending the map $s\mapsto h^{\Xi}_{s,0}(\lambda)$, for $s\leq 0$, and taking the derivate of $G_t(\cdot)$ computed in \eqref{eq:derivada}, we have
	\[\begin{split}
		\mathbb{E}^{(e,1)}\Big[G_t'(\lambda)\mathbf{1}_{\{\overline{\xi}_t< -x\}}\Big]&= z\mathbb{E}^{(e,1)}\left[\exp\left\{-zh_{0,t}(\lambda)-\int_{0}^{t} \psi'_0\big(h_{s,t}(\lambda)\big)\mathrm{d} s\right\}\mathbf{1}_{\{\overline{\xi}_t< -x\}}\right]\\ &= z\mathbf{E}^{(e,1)}\left[\exp\left\{-zh_{-t,0}^{\Xi}(\lambda)-\int_{-t}^{0} \psi'_0\big(h_{s,0}^{\Xi}(\lambda)\big)\mathrm{d} s\right\}\mathbf{1}_{\{\underline{\Xi}_{-t}> x\}}\right],
	\end{split}\]
where we recall that $(\Xi, \mathbf{P}^{(e,1)})$ is the homogeneous L\'evy process indexed in $\mathbb{R}$ associated to  $(\xi, \mathbb{P}^{(e,1)})$. Next, we simplify the notation by  introducing, for  $t\geq 0$,
	\begin{equation*}
		F_t(\lambda):=\exp\left\{-zh_{-t,0}^{\Xi}(\lambda)-\int_{-t}^{0} \psi'_0\big(h_{s,0}^{\Xi}(\lambda)\big)\mathrm{d} s\right\} .
	\end{equation*}
	Hence, making use of the above observations, we deduce 
	\begin{equation*}
		e^{-\Phi_{\xi}(1)t}\mathbb{P}_{(z,x)}\Big(Z_t>0, \ \overline{\xi}_t< 0\Big)  = z\mathbf{P}^{(e,1)}\big(\underline{\Xi}_{-t}> x\big)\int_{0}^{\infty} \mathbf{E}^{(e,1)}\Big[F_t(\lambda)\  \big|\big. \  \underline{\Xi}_{-t}> x\Big]\mathrm{d} \lambda.
	\end{equation*}
	Now, taking into account that
	\begin{equation}\label{eq_compK}
		\lim\limits_{t\to \infty} 	t^{1/2} \mathbf{P}^{(e,1)}\big(\underline{\Xi}_{-t}> x\big) =	\lim\limits_{t\to \infty} 	t^{1/2}\mathbb{P}^{(e,1)}\left(\overline{\xi}_{t}< -x\right) = \sqrt{\frac{2}{\pi \Phi_\xi''(1)}} \mathbb{E}^{(e,1)}\big[H_1\big]U^{(1)}(-x),
	\end{equation} 
	thus the proof of this lemma will be completed  once we have shown
	\begin{equation*}
		\begin{split}
			\lim\limits_{t\to \infty} \int_{0}^{\infty} \mathbf{E}^{(e,1)}\Big[F_t(\lambda)\  \big|\big. \ \underline{\Xi}_{-t} >x \Big]\mathrm{d} \lambda &=\lim\limits_{t\to \infty} \int_{0}^{\infty} \mathbf{E}^{(e,1)}_{-x}\Big[F_t(\lambda)\  \big|\big. \ \underline{\Xi}_{-t} >0 \Big]\mathrm{d} \lambda\\  &=  \mathbf{E}^{(e,1), \uparrow}_{-x}\left[\int_{0}^{\infty}\exp\left\{-\int_{-\infty}^{0}\psi_0'\big(h_{s,0}^{\Xi}(\lambda)\big)\mathrm{d} s\right\}\mathrm{d} \lambda\right]\\ &=:b(x).
		\end{split}
	\end{equation*}
	
	The arguments used to deduce the preceding limit are quite involved, for that reason we split its proof in three steps.
	\\
	
	\textit{Step 1.} Let us first introduce the following functions, for $r, \lambda\geq 0$ and $t\geq 0$,
	$$ f_r(t,\lambda):= \mathbf{E}^{(e,1)}_{-x}\Big[F_t(\lambda)\  \big|\big. \  \underline{\Xi}_{-(t+r)}> 0 \Big] ,$$
	and
	$$ g_r(t,\lambda):= \mathbf{E}^{(e,1)}_{-x}\left[\exp\left\{-\int_{-t}^{0}\psi_0'(h_{s,0}^{\Xi}(\lambda)) \mathrm{d} s \right\}  \Big|\Big. \  \underline{\Xi}_{-(t+r)}> 0 \right].$$

	Since $F_t(\lambda)$ and $\exp\{-\int_{-t}^{0}\psi_0'(h_{s,0}^{\Xi}(\lambda))\mathrm{d} \lambda\}$ are bounded random variables, we may deduce 
	\begin{equation}\label{eq:convfr}
		f_r(t,\lambda) \to\mathbf{E}^{(e,1), \uparrow}_{-x}\left[F_t(\lambda)\right]  \quad \text{and} \quad g_r(t,\lambda) \to \mathbf{E}^{(e,1), \uparrow}_{-x}\left[\exp\left\{-\int_{-t}^{0}\psi_0'(h_{s,0}^{\Xi}(\lambda)) \mathrm{d} s \right\} \right],
	\end{equation}
	as $r\to \infty$. 
	We first prove the convergence for $F_t(\lambda)$, since the same arguments will lead to the other convergence. Appealing to the Markov property,  we have, 
		\begin{equation}\label{eq_Rs}
		\mathbf{E}_{-x}^{(e,1)}\Big[F_t(\lambda) \  |\   \underline{\Xi}_{t+r} > 0\Big]  =  \mathbf{E}_{-x}^{(e,1)} \left[F_t(\lambda) \frac{\mathbf{P}^{(e,1)}_{-\Xi_{-t}}\big(\underline{\Xi}_{-r}> 0\big)}{\mathbf{P}^{(e,1)}_{-x}\big(\underline{\Xi}_{-(t+r)}>0\big)} \mathbf{1}_{\{\underline{\Xi}_{-t}>0\}} \right].
	\end{equation}
	Now since \eqref{eq_compK} holds, we have for $\epsilon >0$ that there exists a constant  $N_1>0$ (which depends on $\epsilon$) such that the following inequality is satisfied for all $r \geq N_1$,
	\begin{equation*}
		\frac{\mathbf{P}^{(e,1)}_{-\Xi_{-r}}\big(\underline{\Xi}_{-
				r}>0\big)}{\mathbf{P}^{(e,1)}_{-x}\big(\underline{\Xi}_{-(t+r)}>0 \big)} \leq \frac{(1+\epsilon)}{(1-\epsilon)} \left(\frac{r}{t+r}\right)^{-1/2} \frac{U^{(1)}(-\Xi_{-t})} {U^{(1)}(-x)}.
	\end{equation*}
	Therefore, we deduce that for $r \geq N_1$,
	\begin{equation}\label{eq_lowerboundP}
		\frac{\mathbf{P}^{(e,1)}_{-\Xi_t}\big(\underline{\Xi}_{-
				r}>0\big)}{\mathbf{P}^{(e,1)}_{-x}\big(\underline{\Xi}_{-(t+r)}>0 \big)} \leq \frac{(1+\epsilon)}{(1-\epsilon)} \left(1+ \frac{t}{N}\right)^{1/2} \frac{U^{(1)}(-\Xi_{-t})}{U^{(1)}(-x)}.
	\end{equation} 
	Since the renewal function $U^{(1)}$ satisfies $$\mathbf{E}_{-x}^{(e,1)}\big[U^{(1)}(-\Xi_{-t})\mathbf{1}_{\{\underline{\Xi}_{-t}>0\}}\big] = \mathbb{E}^{(e,1)}_x\big[U^{(1)}(-\xi_t)\mathbf{1}_{\{\overline{\xi}_t <0\}}\big]=U^{(1)}(-x)$$ and $U^{(1)}(-x)$ is finite and $F_t(\lambda)$ is a bounded random variable, then we can apply the Dominated Convergence Theorem in \eqref{eq_Rs} to obtain the first convergence in \eqref{eq:convfr}. 
	
	Similarly, inequality \eqref{eq_lowerboundP} implies that the following upper bound also holds
\[\begin{split}
	g_r(t,\lambda) \leq C_1(t) \mathbf{E}^{(e,1), \uparrow}_{-x}\left[\exp\left\{-\int_{-t}^{0}\psi_0'(h_{s,0}^{\Xi}(\lambda)) \mathrm{d} s \right\} \right],
\end{split}\]	
where $C_1(t)$ is a positive constant which depends on $t$. We may now appeal to the Dominated Convergence Theorem together with our hypothesis \eqref{eq:assumpinter}, to deduce that for $t\geq 1$
\begin{equation*}
	\int_{0}^{\infty} g_r(t,\lambda) \mathrm{d}\lambda \to \int_{0}^{\infty}  \mathbf{E}^{(e,1), \uparrow}_{-x}\left[\exp\left\{-\int_{-t}^{0}\psi_0'(h_{s,0}^{\Xi}(\lambda)) \mathrm{d} s \right\} \right] \mathrm{d} \lambda, \quad \text{as}\quad  r \to \infty.
\end{equation*}
Furthermore, since $f_r(t,\lambda)\leq g_r(t,\lambda)$, an application of  the \textit{generalised} Dominated Convergence Theorem (see for instance Folland \cite[ Exercise 2.20]{folland}) give us  
\begin{equation}\label{eq_tyh}
	\lim\limits_{r\to \infty}	\int_{0}^{\infty} \mathbf{E}^{(e,1)}_{-x}\Big[F_t(\lambda)\  \big|\big. \  \underline{\Xi}_{-(t+r)}> 0 \Big]\mathrm{d} \lambda = \int_{0}^{\infty}\mathbf{E}^{(e,1), \uparrow}_{-x}\left[F_t(\lambda)\right] \mathrm{d} \lambda.
\end{equation}
\\

\textit{Step 2}. Let $1\leq s\leq t$, $\lambda \geq 0$ and $\gamma \in (1,2]$. From the proof of Lemma 4 in \cite{bansaye2021extinction}, we can deduce
\begin{equation*}
	\left| \mathbf{E}_{-x}^{(e,1)} \Big[F_t(\lambda) - F_s(\lambda)\ \big|\big. \  \underline{\Xi}_{-\gamma t}> 0\Big] \right|  \leq  	C_2 \mathbf{E}_{-x}^{(e,1),\uparrow}\Big[| F_t(\lambda) - F_s(\lambda)| \Big] ,
\end{equation*}
where $C_2$ is a positive constant. Hence
\begin{equation*}
	\left| \int_{0}^{\infty} \mathbf{E}_{-x}^{(e,1)} \Big[F_t(\lambda) - F_s(\lambda)\ \big|\big. \  \underline{\Xi}_{-\gamma t}> 0\Big] \mathrm{d}\lambda\right|  \leq  	C_2 \int_{0}^{\infty}\mathbf{E}_{-x}^{(e,1),\uparrow}\Big[| F_t(\lambda) - F_s(\lambda)| \Big] \mathrm{d} \lambda.
\end{equation*}
Now, under the event that $\{\underline{\Xi}_{-\gamma t}> 0\}$, we know  that, for each $\lambda\ge 0$, the following inequalities $h_{-s,0}^{\Xi}(\lambda)\leq \lambda e^{-\Xi_{-s}}\leq \lambda$ hold. In particular, under  $\{\underline{\Xi}_{-\gamma t}> 0\}$, we obtain 
\[\begin{split}
	| F_t(\lambda) - F_s(\lambda)|  &= \exp\left\{-\int_{-s}^{0} \psi_0'(h_{u,0}^{\Xi}(\lambda))\mathrm{d}u\right\}\\ & \hspace{1.7cm} \left|\exp\left\{-zh_{-t,0}^{\Xi}(\lambda)-\int_{-t}^{-s} \psi'_0\big(h_{u,0}^{\Xi}(\lambda)\big)\mathrm{d} u\right\}-\exp\{-zh_{-s,0}^{\Xi}(\lambda)\}\right| \\&  \leq \exp\left\{-\int_{-s}^{0} \psi_0'(h_{u,0}^{\Xi}(\lambda))\mathrm{d}u\right\} \left|\exp\left\{-\int_{-t}^{-s} \psi'_0\big(h_{u,0}^{\Xi}(\lambda)\big)\mathrm{d} u\right\} - e^{-z\lambda} \right|\\ &\leq 2 \exp\left\{-\int_{-1}^{0} \psi_0'(h_{u,0}^{\Xi}(\lambda))\mathrm{d}u\right\}.
\end{split}\]
It then follows, from the previous calculations and our assumption \eqref{eq:assumpinter} together with the Dominated Convergence Theorem, that
\begin{equation*}
	\lim\limits_{s\to \infty}\lim\limits_{t\to \infty}	\left| \int_{0}^{\infty}\mathbf{E}_{-x}^{(e,1)} \Big[F_t(\lambda) - F_s(\lambda)\ \big|\big. \  \underline{\Xi}_{-\gamma t}> 0\Big] \mathrm{d} \lambda \right| = 0,
\end{equation*}	
which in particular yields
\begin{equation*}
	\lim\limits_{s\to \infty}\lim\limits_{t\to \infty}	\int_{0}^{\infty} \mathbf{E}_{-x}^{(e,1)} \Big[F_t(\lambda) - F_s(\lambda)\ \big|\big. \  \underline{\Xi}_{-\gamma t}> 0\Big] \mathrm{d} \lambda= 0.
\end{equation*}	
Thus, appealing to \eqref{eq_tyh} in Step 1, we get
\begin{eqnarray*}
	\lim\limits_{t\to \infty} \int_{0}^{\infty} \mathbf{E}_{-x}^{(e,1)} \Big[F_t(\lambda)\ \big|\big. \  \underline{\Xi}_{-\gamma t}> 0\Big] \mathrm{d} \lambda &=& \lim\limits_{s\to \infty}	\lim\limits_{t\to \infty} \int_{0}^{\infty}  \mathbf{E}_{-x}^{(e,1)} \Big[ F_s(\lambda)\ \big|\big. \  \underline{\Xi}_{-\gamma t}> 0\Big] \mathrm{d} \lambda\nonumber \\ &=& \lim\limits_{s\to \infty} \int_{0}^{\infty}\mathbf{E}^{(e,1), \uparrow}_{-x}\left[F_s(\lambda)\right] \mathrm{d} \lambda.
\end{eqnarray*}
In order to deal with the above limit in the right-hand side, first note that  $h_{-s,0}^{\Xi}(\lambda)\leq \lambda e^{-\Xi_{-s}} \to 0$, as $s\to \infty$, $\mathbf{P}^{(e,1)}_{-x}$-a.s. Moreover, we have 
$$ \mathbf{E}^{(e,1), \uparrow}_{-x}\left[F_{s}(\lambda)\right] \to \mathbf{E}^{(e,1), \uparrow}_{-x}\left[\exp\left\{-\int_{-\infty}^{0}\psi_0'(h_{u,0}^{\Xi}(\lambda))\mathrm{d}\lambda\right\}\right],\quad \text{as}\quad s \to \infty,$$
and for $s\geq 1$,
$$  \mathbf{E}^{(e,1), \uparrow}_{-x}\left[F_{s}(\lambda)\right] \leq \mathbf{E}^{(e,1), \uparrow}_{-x}\left[\exp\left\{-\int_{-1}^{0}\psi_0'(h_{u,0}^{\Xi}(\lambda))\mathrm{d}\lambda\right\}\right].$$
Hence, we may now apply once again the Dominated Convergence Theorem to deduce that 
\[\int_{0}^{\infty}\lim\limits_{s\to \infty}  \mathbf{E}^{(e,1), \uparrow}_{-x}\left[F_s(\lambda)\right] \mathrm{d} \lambda =b(x) <\infty.\]
In other words, we have
\[\lim\limits_{t\to \infty} \int_{0}^{\infty} \mathbf{E}_{-x}^{(e,1)} \Big[F_t(\lambda)\ \big|\big. \  \underline{\Xi}_{-\gamma t}> 0\Big] \mathrm{d} \lambda = b(x).\]
Next, from \eqref{eq_compK} we obtain 
\[\begin{split}
	\lim\limits_{t\to \infty} \frac{1}{\mathbf{P}_{-x}^{(e,1)}(\underline{\Xi}_{-t}> 0)}& \int_{0}^{\infty} \mathbf{E}_{-x}^{(e,1)} \Big[F_t(\lambda)\mathbf{1}_{\{\underline{\Xi}_{-\gamma t}> 0\}}\Big] \mathrm{d} \lambda\\ &= \lim\limits_{t\to \infty} \frac{\mathbf{P}^{(e,1)}_{-x}\big(\underline{\Xi}_{-\gamma t}> 0\big)}{\mathbf{P}^{(e,1)}_{-x}\big(\underline{\Xi}_{-t} > 0\big)}  \int_{0}^{\infty} \mathbf{E}_{-x}^{(e,1)} \Big[F_t(\lambda)\ \big|\big. \  \underline{\Xi}_{-\gamma t}>0\Big]  \mathrm{d} \lambda\\ &= \gamma^{-1/2} b(x).
\end{split}\]
Since $\gamma$ may be chosen arbitrarily close to 1, we have
\begin{eqnarray*}
	\int_{0}^{\infty} \mathbf{E}^{(e,1)}_{-x}\Big[F_t(\lambda) \mathbf{1}_{\{\underline{\Xi}_{-\gamma t}> 0 \}}\Big]\mathrm{d} \lambda - b(x)\mathbf{P}_{-x}^{(e,1)}\big(\underline{\Xi}_{-t}> 0\big)= o(1)\mathbf{P}_{-x}^{(e,1)}\big(\underline{\Xi}_{-t}> 0\big).
\end{eqnarray*}
\\

\textit{Step 3}. Let $\lambda\geq 0$, $t\geq 1$ and $\gamma \in (1,2]$ and denote 
\[J_t(\lambda):=\frac{1}{\mathbf{P}_{-x}^{(e,1)}(\underline{\Xi}_{-t}> 0)}\mathbf{E}^{(e,1)}_{-x}\Big[F_t(\lambda)\left( \mathbf{1}_{\{\underline{\Xi}_{-t}> 0\}} -  \mathbf{1}_{\{\underline{\Xi}_{ -\gamma t}> 0 \}}\right)\Big].\]
Observe from \eqref{eq_compK} and the fact that $F_t(\lambda)\leq 1$, that the following holds
\[0\leq J_t(\lambda)\leq 1 - \frac{\mathbf{P}^{(e,1)}_{-x}\big(\underline{\Xi}_{-\gamma t}> 0\big)}{\mathbf{P}^{(e,1)}_{-x}\big(\underline{\Xi}_{-t}> 0\big)} \to 1- \gamma^{1/2}, \quad \text{as} \quad t \to \infty. \] 
Since $\gamma$ may be taken arbitrary close to 1,  we deduce that  $J_t(\lambda) \to 0$ as $t\to \infty$.  	In addition, 
\[\begin{split}J_t(\lambda)&\leq \mathbf{E}^{(e,1)}_{-x}\left[\exp\left\{-\int_{-t}^{0}\psi_0'(h_{s,0}^{\Xi}(\lambda)) \mathrm{d} s \right\}\ \Big|\Big. \ \underline{\Xi}_{-t}> 0 \right] \\ & \leq \mathbf{E}^{(e,1)}_{-x}\left[\exp\left\{-\int_{-1}^{0}\psi_0'(h_{s,0}^{\Xi}(\lambda)) \mathrm{d} s \right\}\ \Big|\Big. \ \underline{\Xi}_{-t}> 0 \right]\\ &\leq  C_3 \mathbf{E}^{(e,1), \uparrow}_{-x}\left[\exp\left\{-\int_{-1}^{0}\psi_0'(h_{s,0}^{\Xi}(\lambda)) \mathrm{d} s \right\} \right],\end{split}\]
where $C_3$ is a positive constant and the right-hand side is an integrable function in $\lambda$ thanks to the assumption \eqref{eq:assumpinter}. Hence, appealing again to the Dominate Convergence Theorem, we see  
\begin{equation*}\label{eq_cotaRs}
	\int_{0}^{\infty} \mathbf{E}^{(e,1)}_{-x}\Big[F_t(\lambda)\left( \mathbf{1}_{\{\underline{\Xi}_{ -t}> 0\}} -  \mathbf{1}_{\{\underline{\Xi}_{ -\gamma t}> 0 \}}\right)\Big] \mathrm{d} \lambda = o(1)\mathbf{P}_{-x}^{(e,1)}\big(\underline{\Xi}_{-t}> 0\big).
\end{equation*}
We combine the previous limit  with the conclusion of  Steps 2 to deduce, as promised earlier, that
\[\begin{split}
	\int_{0}^{\infty} \mathbf{E}^{(e,1)}_{-x}\Big[F_t(\lambda)& \mathbf{1}_{\{\underline{\Xi}_{- t}> 0 \}}\Big]\mathrm{d} \lambda - b(x)\mathbf{P}_{-x}^{(e,1)}\big(\underline{\Xi}_{-t}> 0\big)\\  & =	\int_{0}^{\infty} \mathbf{E}^{(e,1)}_{-x}\Big[F_t(\lambda) \mathbf{1}_{\{\underline{\Xi}_{- t}> 0 \}}\Big]\mathrm{d} \lambda- 	\int_{0}^{\infty} \mathbf{E}^{(e,1)}_{-x}\Big[F_t(\lambda) \mathbf{1}_{\{\underline{\Xi}_{-\gamma t}> 0 \}}\Big]\mathrm{d} \lambda \\ & \hspace{2cm}+ 	\int_{0}^{\infty} \mathbf{E}^{(e,1)}_{-x}\Big[F_t(\lambda) \mathbf{1}_{\{\underline{\Xi}_{-\gamma t}> 0 \}}\Big]\mathrm{d} \lambda-  b(x)\mathbf{P}_{-x}^{(e,1)}\big(\underline{\Xi}_{-t}> 0\big)\\ & = o(1)\mathbf{P}_{-x}^{(e,1)}\big(\underline{\Xi}_{-t}> 0\big).
\end{split}\]
Finally, similarly as in the proof of Theorem \ref{theo_stronglycase}, we see that the moment condition \eqref{eq_xlogx} guarantees that $b(x)>0$. This concludes the proof.
\end{proof}

With  Lemmas  \ref{lem_inter_cota0} and \ref{lem_inter_cota1} in hand, we may now proceed to the proof of Theorem \ref{theo_interm} following similar ideas as those used in Theorem 1 in \cite{bansaye2021extinction}.  

\begin{proof}[Proof of Theorem \ref{theo_interm}] 
Let $z,\epsilon >0$ and $x<0$.  From Lemma \ref{lem_inter_cota0}, we have  for every $\delta \in (0,1)$, 
$$\lim\limits_{y \to \infty} \limsup_{t\to \infty} t^{1/2} e^{-t\Phi_{\xi}(\gamma)}\mathbb{P}_{(z,x)}\Big(Z_t>0,\  \overline{\xi}_{t-\delta}\geq y\Big) = 0.$$
Then it follows that, we may choose $y>0$ such that for $t$ sufficiently large 
\begin{equation*}
	\mathbb{P}_{(z,x)}\Big(Z_t>0,\ \overline{\xi}_{t-\delta}\geq y\Big)   \leq \epsilon \mathbb{P}_{(z,x)}\Big(Z_t>0,\  \overline{\xi}_{t-\delta}< y\Big).
\end{equation*} 
Further, since $\{Z_{t} >0\}\subset \{Z_{t-\delta} >0\}$ for $t$ large, we deduce that
\begin{eqnarray*}
	\mathbb{P}_{z}(Z_t>0) &=&  \mathbb{P}_{(z,x)}\Big(Z_t>0,\  \overline{\xi}_{t-\delta}\geq y\Big)  + \mathbb{P}_{(z,x)}\Big(Z_t>0,\  \overline{\xi}_{t-\delta}< y\Big) \\ &\leq & (1+\epsilon) \mathbb{P}_{(z,x-y)}\Big(Z_{t-\delta}>0,\  \overline{\xi}_{t-\delta}< 0\Big).
\end{eqnarray*}
In other words, for every $\epsilon >0$ there exists $y'<0$ such that  
\[\begin{split}
	&(1-\epsilon) t^{1/2}e^{-\Phi_{\xi}(1)t}\mathbb{P}_{(z,y')}\Big(Z_t>0,\  \overline{\xi}_{t} <0\Big) \leq  t^{1/2}e^{-\Phi_{\xi}(1)t}\mathbb{P}_{z}(Z_t>0) \\ &\leq  (1+\epsilon) (t-\delta)^{1/2} e^{-\Phi_{\xi}(1)(t-\delta)}\mathbb{P}_{(z,y')}\Big(Z_{t-\delta}>0,\ \overline{\xi}_{t-\delta} < 0\Big)\frac{t^{1/2}e^{-\Phi_{\xi}(1)t}}{(t-\delta)^{1/2} e^{-\Phi_{\xi}(1)(t-\delta)}}.
\end{split}\]
Now, appealing to Lemma \ref{lem_inter_cota1}, we  have 
\[\begin{split}
	\lim\limits_{t\to\infty} 	t^{1/2} e^{-\Phi_{\xi}(1)t}\mathbb{P}_{(z,y')}\Big(Z_t>0, \ \overline{\xi}_t< 0\Big)    = z  \sqrt{\frac{2}{\pi \Phi_\xi''(1)}} \mathbb{E}^{(e,1)}\big[H_1\big]\mathfrak{b}_2(y'),
\end{split}\]
where 
\begin{equation}
	\mathfrak{b}_2(y')=U^{(1)}(-y') \mathbf{E}^{(1), \uparrow}_{-y'}\left[\int_{0}^{\infty}\exp\left\{-\int_{-\infty}^{0}\psi_0'\big(h_{s,0}^{\Xi}(\lambda)\big)\mathrm{d} s\right\}\mathrm{d} \lambda\right].
\end{equation}
Hence, we obtain 
\[\begin{split}
	(1-\epsilon) z  \sqrt{\frac{2}{\pi \Phi_\xi''(1)}}  \mathbb{E}^{(e,1)}\big[H_1\big]\mathfrak{b}_2(y') \leq \lim\limits_{t \to \infty} & t^{1/2}e^{-t\Phi_{\xi}(1)} \mathbb{P}_{z}(Z_t>0) \\ &\leq (1+\epsilon) z  \sqrt{\frac{2}{\pi \Phi_\xi''(1)}} \mathbb{E}^{(e,1)}\big[H_1\big]\mathfrak{b}_2(y')e^{-\Phi_{\xi}(1)\delta}.
\end{split}\]
On the other hand, we observe that $y'$ is a sequence which may depend on $\epsilon$ . Further, this sequence $y'$ goes to minus  infinity as $\epsilon$ goes to 0.  Then, for any sequence $y'=y_{\epsilon}$, we deduce that 
\[\begin{split}
	0<	(1-\epsilon)z  \sqrt{\frac{2}{\pi \Phi_\xi''(1)}}  \mathbb{E}^{(e,1)}\big[H_1\big]&\mathfrak{b}_2(y_\epsilon)  \leq \lim\limits_{t \to \infty} t^{1/2}e^{-\Phi_{\xi}(1)t}\mathbb{P}_{z}(Z_t>0) \\ & \leq (1+\epsilon) z  \sqrt{\frac{2}{\pi \Phi_\xi''(1)}}  \mathbb{E}^{(e,1)}\big[H_1\big]\mathfrak{b}_2(y_\epsilon) < \infty.
\end{split}\]
Therefore, by  letting $\epsilon\to 0$, we get
\[\begin{split}
	0<	\liminf_{\epsilon \to 0} (1-\epsilon) &z  \sqrt{\frac{2}{\pi \Phi_\xi''(1)}}  \mathbb{E}^{(e,1)}\big[H_1\big]\mathfrak{b}_2(y_\epsilon)  \leq\lim\limits_{t \to \infty} t^{1/2}e^{-\Phi_{\xi}(1)t}\mathbb{P}_{z}(Z_t>0)\\  &\leq  \limsup_{\epsilon \to 0}(1+\epsilon)  z  \sqrt{\frac{2}{\pi \Phi_\xi''(1)}}  \mathbb{E}^{(e,1)}\big[H_1\big]\mathfrak{b}_2(y_\epsilon) e^{-\Phi_{\xi}(1)\delta }< \infty.
\end{split}\]
Since $\delta$ can be taken arbitrary close to 0,  we deduce 
\begin{equation*}
	\lim\limits_{t \to \infty} t^{1/2}e^{-\Phi_{\xi}(1)t} \mathbb{P}_{z}(Z_t>0)=z  \sqrt{\frac{2}{\pi \Phi_\xi''(1)}}  \mathbb{E}^{(e,1)}\big[H_1\big]\mathfrak{B}_2,
\end{equation*}
where
$$ \mathfrak{B}_2:=\lim\limits_{\epsilon \to 0} \mathfrak{b}_2(y_\epsilon)= \lim_{\epsilon \to 0} U^{(1)}(-y_\epsilon) \mathbf{E}^{(1), \uparrow}_{-y_\epsilon}\left[\int_{0}^{\infty}\exp\left\{-\int_{-\infty}^{0}\psi_0'\big(h_{s,0}^{\Xi}(\lambda)\big)\mathrm{d} s\right\}\mathrm{d} \lambda\right]. $$
The proof is now complete.

\end{proof}
\subsection{The $Q$ process}

\begin{proof}[Proof of Theorem \ref{teo:Qprocess}]
	We first prove part (i). We only deduce it for the strongly subcritical regime, for the intermediate subcritical regime the arguments are basically the same. Let $z, t>0$ and $\Lambda\in \mathcal{F}_t$, from the Markov property, we obtain
	\begin{equation*}
		\mathbb{P}_z\big(\Lambda\ |\ T_0>t+s\big) = \mathbb{E}_z\left[\mathbf{1}_{\{\Lambda, T_0>t\}}\frac{\mathbb{P}_{Z_t}(T_0>s)}{\mathbb{P}_z(T_0>t+s)}\right].
	\end{equation*}
	From Theorem \ref{theo_stronglycase}, for any $\epsilon>0$ and $t$ large enough, we deduce that
	\begin{equation*}
		\frac{\mathbb{P}_{Z_t}(T_0>s)}{\mathbb{P}_z(T_0>t+s)} = \frac{e^{-\Phi_\xi(1)s}\mathbb{P}_{Z_t}(Z_s>0)e^{-\Phi_\xi(1)t}}{e^{-\Phi_\xi(1)(t+s)}(Z_{t+s}>0)}\leq e^{-\Phi_\xi(1)t}\left(\frac{\epsilon + Z_t \mathfrak{B}_1}{-\epsilon + z \mathfrak{B}_1}\right).
	\end{equation*}
	Further, from \eqref{martingquenched}, we have  
	\begin{equation*}
		e^{-\Phi_\xi(1)t} \mathbb{E}_z[Z_t\ |\ S] = z e^{\xi_t - \Phi_\xi(1)t},
	\end{equation*}
	where the random variable in the right-hand side above is integrable thanks to our exponential moment condition \eqref{eq_moments} with $\vartheta=1$. Hence,  the Dominated Convergence Theorem implies that
	\begin{equation*}
		\begin{split}
			\lim\limits_{s\to \infty} 	\mathbb{P}_z\big(\Lambda\ |\ T_0>t+s\big) &= \mathbb{E}_z\left[\mathbf{1}_{\{\Lambda, T_0>t\}}\lim\limits_{s\to \infty} \frac{\mathbb{P}_{Z_t}(T_0>s)}{\mathbb{P}_z(T_0>t+s)}\right] \\ &= \mathbb{E}_z\left[\mathbf{1}_{\{\Lambda, T_0>t\}}\frac{Z_t}{z} e^{-\Phi_\xi(1)t}\right] = \mathbb{E}_z\left[\frac{Z_t}{z} e^{-\Phi_\xi(1)t}\mathbf{1}_{\Lambda}\right].
		\end{split}
	\end{equation*}
	We now prove part (ii). The fact that the process $(e^{-\Phi_\xi(1)t}Z_t, \ t\geq 0)$ is a martingale follows directly from  \eqref{martingquenched} by applying the Markov property as follows: for $0\leq s\leq t$, 
	\begin{equation*}
		\mathbb{E}_z\big[e^{-\Phi_\xi(1)(t+s)}Z_{t+s}\  |\  \mathcal{F}_{s}\big]= e^{-\Phi_\xi(1)(t+s)}\mathbb{E}_{Z_s}[Z_t] = e^{-\Phi_\xi(1)(t+s)} Z_s \mathbb{E}[e^{\xi_t}] =  e^{-\Phi_\xi(1) s} Z_s,
	\end{equation*}
	which establishes the martingale property.

To deduce part (iii), we compute the Laplace transform of $Z$, under $\mathbb{P}^\natural_z$. Fix $z, t>0$ and $\lambda \geq 0$,  using  \eqref{eq_transLaplace} and part (ii), we get
	\begin{equation*}
		\begin{split}
			\mathbb{E}^{\natural}_z\Big[e^{-\lambda Z_t}\Big] &= \mathbb{E}_z\left[\frac{Z_t}{z}e^{-\Phi_\xi(1)t}e^{-\lambda Z_t}\right] = -\frac{e^{-\Phi_\xi(1)t}}{z} \frac{\mathrm{d}}{\mathrm{d} \lambda}\mathbb{E}_z\Big[e^{-\lambda Z_t}\Big]\\ &= -\frac{e^{-\Phi_\xi(1)t}}{z} \frac{\mathrm{d}}{\mathrm{d} \lambda}\mathbb{E}^{(e)}\Big[\exp\{-zh_{0,t}(\lambda)\}\Big].
		\end{split}
	\end{equation*}
	Now, note that
	\begin{equation}
		\frac{\mathrm{d}}{\mathrm{d} \lambda}\mathbb{E}^{(e)}\Big[\exp\{-zh_{0,t}(\lambda)\}\Big] =- z\mathbb{E}^{(e)}\left[\exp\{-zh_{0,t}(\lambda)\}e^{\xi_t}\frac{\mathrm{d} }{\mathrm{d}u}v_t(0, u,\xi)\Big|\Big._{u=\lambda e^{\xi_t}}\right].
	\end{equation}
	Thus, from \eqref{eq:derivatev}, we deduce 
	\begin{equation*}
		\begin{split}
			\mathbb{E}^{\natural}_z\Big[e^{-\lambda Z_t}\Big] &= e^{-\Phi_\xi(1)t} \mathbb{E}^{(e)}\left[\exp\{-zh_{0,t}(\lambda)\}e^{\xi_t}\frac{\mathrm{d} }{\mathrm{d}u}v_t(0, u,\xi)\Big|\Big._{u=\lambda e^{\xi_t}}\right]\\ &= e^{-\Phi_\xi(1)t}\mathbb{E}^{(e)}\left[e^{\xi_t}\exp\left\{-zh_{0,t}(\lambda) -\int_{0}^{t}\psi_0'(h_{s,t}(\lambda))\mathrm{d} s\right\}\right] \\ &= \mathbb{E}^{(e,1)}\left[\exp\left\{-zh_{0,t}(\lambda) -\int_{0}^{t}\psi_0'(h_{s,t}(\lambda))\mathrm{d} s\right\}\right],
		\end{split}
	\end{equation*}
	where in the last equality we have used the definition of the Esscher transform given in \eqref{eq_medida_Essher}. This completes the proof. 
\end{proof}

\textbf{Acknowledgements}:  N.C.-T. acknowledges support from CONACyT-MEXICO grant no. 636133.

\bibliographystyle{abbrv}
\bibliography{references}

\begin{thebibliography}{10}

\bibitem{afanasyev1980}
V.~I. Afanasyev.
\newblock Limit theorems for a conditional random walk and some applications.
\newblock {\em Diss. Cand. Sci., MSU, Moscow,}, 1980.

\bibitem{afanasyev2014conditional}
V.~I. Afanasyev, C.~B\"{o}inghoff, G.~Kersting, and V.~A. Vatutin.
\newblock Conditional limit theorems for intermediately subcritical branching
  processes in random environment.
\newblock {\em Ann. Inst. Henri Poincar\'{e} Probab. Stat.}, 50(2):602--627,
  2014.

\bibitem{afanasyevstrongly}
V.~I. Afanasyev, J.~Geiger, G.~Kersting, and V.~A. Vatutin.
\newblock Functional limit theorems for strongly subcritical branching
  processes in random environment.
\newblock {\em Stochastic Process. Appl.}, 115(10):1658--1676, 2005.

\bibitem{bansaye2021extinction}
V.~Bansaye, J.~C. Pardo, and C.~Smadi.
\newblock Extinction rate of continuous state branching processes in critical
  {L}\'{e}vy environments.
\newblock {\em ESAIM Probab. Stat.}, 25:346--375, 2021.

\bibitem{bansaye2013extinction}
V.~Bansaye, J.~C. Pardo~Millan, and C.~Smadi.
\newblock On the extinction of continuous state branching processes with
  catastrophes.
\newblock {\em Electron. J. Probab.}, 18:no. 106, 31, 2013.

\bibitem{bansaye2015scaling}
V.~Bansaye and F.~Simatos.
\newblock On the scaling limits of {G}alton-{W}atson processes in varying
  environments.
\newblock {\em Electron. J. Probab.}, 20:no. 75, 36, 2015.

\bibitem{bertoin1996levy}
J.~Bertoin.
\newblock {\em L\'{e}vy processes}, volume 121 of {\em Cambridge Tracts in
  Mathematics}.
\newblock Cambridge University Press, Cambridge, 1996.

\bibitem{boeinghoff2011branching}
C.~B\"{o}inghoff and M.~Hutzenthaler.
\newblock Branching diffusions in random environment.
\newblock {\em Markov Process. Related Fields}, 18(2):269--310, 2012.

\bibitem{chaumont1996conditionings}
L.~Chaumont.
\newblock Conditionings and path decompositions for {L}\'{e}vy processes.
\newblock {\em Stochastic Process. Appl.}, 64(1):39--54, 1996.

\bibitem{chaumont2005levy}
L.~Chaumont and R.~A. Doney.
\newblock On {L}\'{e}vy processes conditioned to stay positive.
\newblock {\em Electron. J. Probab.}, 10:no. 28, 948--961, 2005.

\bibitem{dekking1987}
F.~M. Dekking.
\newblock On the survival probability of a branching process in a finite state
  i.i.d. environment.
\newblock {\em Stochastic Process. Appl.}, 27(1):151--157, 1987.

\bibitem{doney2007fluctuation}
R.~A. Doney.
\newblock {\em Fluctuation theory for {L}\'{e}vy processes}, volume 1897 of
  {\em Lecture Notes in Mathematics}.
\newblock Springer, Berlin, 2007.

\bibitem{folland}
G.~B. Folland.
\newblock {\em Real analysis}.
\newblock Pure and Applied Mathematics (New York). John Wiley \& Sons, Inc.,
  New York, 1984.

\bibitem{geiger2003limit}
J.~Geiger, G.~Kersting, and V.~A. Vatutin.
\newblock Limit theorems for subcritical branching processes in random
  environment.
\newblock {\em Ann. Inst. H. Poincar\'{e} Probab. Statist.}, 39(4):593--620,
  2003.

\bibitem{grey1974asymptotic}
D.~R. Grey.
\newblock Asymptotic behaviour of continuous time, continuous state-space
  branching processes.
\newblock {\em J. Appl. Probability}, 11:669--677, 1974.

\bibitem{he2018continuous}
H.~He, Z.~Li, and W.~Xu.
\newblock Continuous-state branching processes in {L}\'{e}vy random
  environments.
\newblock {\em J. Theoret. Probab.}, 31(4):1952--1974, 2018.

\bibitem{hirano2001levy}
K.~Hirano.
\newblock L\'{e}vy processes with negative drift conditioned to stay positive.
\newblock {\em Tokyo J. Math.}, 24(1):291--308, 2001.

\bibitem{kurtz1978diffusion}
T.~G. Kurtz.
\newblock Diffusion approximations for branching processes.
\newblock In {\em Branching processes ({C}onf., {S}aint {H}ippolyte, {Q}ue.,
  1976)}, volume~5 of {\em Adv. Probab. Related Topics}, pages 269--292.
  Dekker, New York, 1978.

\bibitem{kyprianou2014fluctuations}
A.~E. Kyprianou.
\newblock {\em Fluctuations of {L}\'{e}vy processes with applications}.
\newblock Universitext. Springer, Heidelberg, second edition, 2014.

\bibitem{lambert2007}
A.~Lambert.
\newblock Quasi-stationary distributions and the continuous-state branching
  process conditioned to be never extinct.
\newblock {\em Electron. J. Probab.}, 12:no. 14, 420--446, 2007.

\bibitem{li2018asymptotic}
Z.~Li and W.~Xu.
\newblock Asymptotic results for exponential functionals of {L}\'{e}vy
  processes.
\newblock {\em Stochastic Process. Appl.}, 128(1):108--131, 2018.

\bibitem{palau2017continuous}
S.~Palau and J.~C. Pardo.
\newblock Continuous state branching processes in random environment: the
  {B}rownian case.
\newblock {\em Stochastic Process. Appl.}, 127(3):957--994, 2017.

\bibitem{palau2018branching}
S.~Palau and J.~C. Pardo.
\newblock Branching processes in a {L}\'{e}vy random environment.
\newblock {\em Acta Appl. Math.}, 153:55--79, 2018.

\bibitem{palau2016asymptotic}
S.~Palau, J.~C. Pardo, and C.~Smadi.
\newblock Asymptotic behaviour of exponential functionals of {L}\'{e}vy
  processes with applications to random processes in random environment.
\newblock {\em ALEA Lat. Am. J. Probab. Math. Stat.}, 13(2):1235--1258, 2016.

\bibitem{patie2018bernstein}
P.~Patie and M.~Savov.
\newblock Bernstein-gamma functions and exponential functionals of {L}\'{e}vy
  processes.
\newblock {\em Electron. J. Probab.}, 23:Paper No. 75, 101, 2018.

\bibitem{ken1999levy}
K.-i. Sato.
\newblock {\em L\'{e}vy processes and infinitely divisible distributions},
  volume~68 of {\em Cambridge Studies in Advanced Mathematics}.
\newblock Cambridge University Press, Cambridge, 2013.

\bibitem{xu2021}
W.~Xu.
\newblock {Asymptotic results for heavy-tailed Lévy processes and their
  exponential functionals}.
\newblock {\em Bernoulli}, 27(4):2766 -- 2803, 2021.

\end{thebibliography}

\end{document}